\documentclass[10pt]{amsart}

\usepackage{amsfonts}
\usepackage{amssymb}
\usepackage{amsmath}
\usepackage{amsthm}
\usepackage{hyperref}
\usepackage{wasysym}
\usepackage[all]{xy} \newdir{ >}{{}*!/-10pt/@{>}}
\usepackage{mathrsfs}
\usepackage{dsfont}
\usepackage{enumerate}
\usepackage{tikz}\usetikzlibrary{decorations.markings}
\usepackage[top=3cm,bottom=3cm,left=3cm,right=3cm]{geometry}
\usepackage{lmodern}
\usepackage[T1]{fontenc}
\usepackage{mathtools}
\usepackage{caption}
\usepackage[url=false,doi=false,style=alphabetic,maxbibnames=10,maxcitenames=5,backend=bibtex8]{biblatex}
\bibliography{../bibliography/library}
\appto{\bibsetup}{\emergencystretch=0.5em}

\newcommand{\N}{\mathds N}
\newcommand{\Z}{\mathds Z}
\newcommand{\Q}{\mathds Q}

\newcommand{\C}{\mathds C}

\newcommand{\id}{\mathds{1}}

\let\oldbigwedge\bigwedge
\renewcommand{\bigwedge}{\oldbigwedge\nolimits}

\newcommand{\chuse}[2]{\left[\begin{array}{c} #1 \\ #2\end{array}\right]}

\newtheorem{thm}{Theorem}[section]
\newtheorem{lem}[thm]{Lemma}
\newtheorem{cor}[thm]{Corollary}\newtheorem{prop}[thm]{Proposition}
\theoremstyle{definition}\newtheorem*{definition}{Definition}\newtheorem{remark}[thm]{Remark}

\newtheorem*{acknowledgement}{Acknowledgement}
\newtheorem*{thm*}{Theorem}

\numberwithin{equation}{subsection}

\address{Dept of Mathematical Sciences\\ Durham University\\ Science Laboratories\\ South Rd.\\ Durham\\ DH1 3LE\\ UK}
\email{jonathan.grant@durham.ac.uk}

\begin{document}
\title{A generators and relations description of a representation category of $U_q(\mathfrak{gl}(1|1))$}
\author{Jonathan Grant}
\begin{abstract}
We use the technique of quantum skew Howe duality to investigate the monoidal category generated by exterior powers of the standard representation of $U_q(\mathfrak{gl}(1|1))$. This produces a complete diagrammatic description of the category in terms of trivalent graphs, with the usual MOY relations plus one additional family of relations. The technique also gives a useful connection between a system of symmetries on $\bigoplus_m \dot{U}_q(\mathfrak{gl}(m))$ and the braiding on the category of $U_q(\mathfrak{gl}(1|1))$-representations which can be used to construct the Alexander polynomial and coloured variants.
\end{abstract}
\maketitle

\section{Introduction}
The Alexander polynomial of a knot is a classical invariant, first studied in the 1920's. It assigns to any knot a Laurent polynomial and provides a tool for telling knots apart. Further exploration shows that in fact the Alexander polynomial contains useful topological information, for instance a famous result of Freedman tells us that if a knot $K$ has the same Alexander polynomial as the unknot, then it must bound a topologically flat embedded disc in the 4-ball.

The Alexander polynomial can be defined in many different ways. One way is to define it by taking the cyclic covering space of the knot complement, giving the rational homology the structure of a $\Q[t,t^{-1}]$-module by using the action of the group of covering transformations and then defining the Alexander polynomial to be the order of the torsion of this module. An advantage of this definition is that it does not require one to choose a diagram for the knot, and hence evidently gives a knot invariant.

However, in \cite{Reshetikhin1990}, Reshetikhin and Turaev define a large family of knot polynomials using deformations of the universal enveloping algebra of a simple Lie algebra, called quantum groups. One colours each strand of a framed link by a representation of the quantum group, and uses certain canonical maps at each crossing and Morse critical point to define an endomorphism of the trivial representation $\C(q)$, which is given by a Laurent polynomial. The definition requires choosing a diagram in Morse position (so that critical points are isolated), but turns out to be independent of this choice, and so gives an invariant of oriented framed links.

It is therefore interesting to study the Alexander polynomial in the context of Reshetikhin and Turaev's polynomials, which are known as quantum knot polynomials. Perhaps one difficulty of such a description is that the Alexander polynomial of a split link is always $0$, whereas the Reshetikhin and Turaev procedure always assigns to a disjoint union of unknots the product of the quantum dimensions of the representations they are coloured with. This therefore hints that the correct representation to look at will have to have quantum dimension $0$. This is only possible by investigating either a quantum universal enveloping algebra with the value $q$ taken to be a root of unity, or by slightly extending the work of Reshetikin and Turaev to superalgebras, where the extra $\Z/2\Z$ grading provides a setting where the quantum dimension can be $0$. A thorough description of this is given in Sartori \cite{Sartori2013}.

The description of the Reshetikhin-Turaev procedure applied to the quantum group $U_q(\mathfrak{sl}_n)$ is greatly simplified by the diagram calculus appearing in the paper of Murakami, Ohtsuki and Yamada \cite{Murakami1998}. This converts a knot or link into a sum of trivalent graphs, and gives relations which simplify any such graph to a linear combination of circles that can be evaluated to polynomials. This provides a technique to calculate the quantum $\mathfrak{sl}_n$ polynomials for links with components coloured by representations $\bigwedge^i_q \C^n_q$, which are exterior powers of the standard $n$-dimensional representation of $U_q(\mathfrak{sl}_n)$.

In a recent paper, Cautis, Kamnitzer and Morrison \cite{Cautis2012} show that in fact the diagram calculus given by Murakami, Ohtsuki and Yamada (MOY) gives a complete description of a certain category of representations of $U_q(\mathfrak{sl}_n)$. That is, if one defines a category $\operatorname{Rep}(\mathfrak{sl}_n)$ to be the category of tensor products of exterior powers of the standard representation of $\mathfrak{sl}_n$, then by interpreting certain trivalent graphs to be maps in $\operatorname{Rep}(\mathfrak{sl}_n)$, one can actually describe every morphism in $\operatorname{Rep}(\mathfrak{sl}_n)$ by MOY diagrams.

The particular interest of this work is the use of skew Howe duality to obtain the result. Cautis, Kamnitzer and Morrison prove a quantum version of skew Howe duality to construct a full and faithful functor between a subcategory of $\operatorname{Rep}(\mathfrak{sl}_n)$ and a quotient of the quantum group $\dot U_q(\mathfrak{gl}_m)$ (thought of as a category) for each $m>2$. This gives the authors a concrete description of $\operatorname{Rep}(\mathfrak{sl}_n)$, which had previously been hard to obtain. An interesting feature of this relationship is that the natural system of symmetries in $\bigoplus_m U_q(\mathfrak{gl}_m)$ actually translates into the braiding on  $\operatorname{Rep}(\mathfrak{sl}_n)$ which is crucial for the definition of the knot polynomials. This relationship was particular exploited in \cite{Lauda2012} and \cite{Queffelec}, which categorifies the features of the above technique and allows Lauda, Queffelec and Rose to provide a way that homology theories defined by Khovanov \cite{Khovanov1999} and Khovanov and Rozansky \cite{Khovanov2004} arise from categorified quantum groups.

In this paper we apply the technique of quantum skew Howe duality to the quantum group $U_q(\mathfrak{gl}(1|1))$ to deduce a similar result about a category of representations of $U_q(\mathfrak{gl}(1|1))$. Defining the category $\operatorname{Rep}$ to be the category with objects given by tensor products of exterior powers of the standard representation, and morphisms given by morphisms of $U_q(\mathfrak{gl}(1|1))$-modules, we establish a diagram calculus for $\operatorname{Rep}$ similar to the one above. The diagrams take the form of ladders, introduced in \cite{Cautis2012}, which are very similar to MOY diagrams but slightly more `rigid' in the sense that two ladders differing only by an isotopy are not always identified. Forming a certain category $\operatorname{Lad}^\Xi_m$ of ladders with $m$ `uprights', we establish the following (Corollary \ref{co:threecategories}):

\begin{thm*}
There is an equivalence of categories
\[\bigoplus_m \operatorname{Lad}^\Xi_m \to \operatorname{Rep}.\]
\end{thm*}

There are a few qualitative differences between this result and the one established in \cite{Cautis2012}. One difference is that due to the superalgebra structure, the $i$th exterior power of the standard representation is never the $0$ representation for $i>0$, which means that trivalent graphs of all colours are necessary, whereas in the case of $U_q(\mathfrak{sl}(n))$ a strand with high enough colours would cause the diagram to evaluate identically to $0$. Related to this is the fact that the dual representation to the standard one never occurs as an exterior power of the standard representation, unlike in the case of $U_q(\mathfrak{sl}_n)$ where $\bigwedge^{n-1}_q \C^n_q \cong (\C^n_q)^*$. The downside to this is that the maps needed to decompose a closed MOY diagram are not contained in the category $\operatorname{Rep}$. A final difference is that not all relations in $\operatorname{Rep}$ are derived from MOY relations: there is an extra non-MOY relation $\Xi$ that plays no part in the evaluation of closed diagrams, but does exist as a relation in $\operatorname{Rep}$.

A similar description of $\operatorname{Rep}$ was obtained independently in the work of Sartori \cite[Theorem 3.3.12]{Sartori2014a}, instead exploiting Schur-Weyl duality. It is nevertheless useful to have a skew Howe duality description of $\operatorname{Rep}$, since skew Howe duality has been fruitfully exploited in many recent papers and seems to be closely related to foam-based descriptions of Khovanov-Rozansky homology.

The intention of this work is to set a possible grounding for a categorification in the style of \cite{Lauda2012} and \cite{Queffelec}, or of \cite{Cautis2010} and \cite{Cautis2012a} which would provide a homology theory categorifying the Alexander polynomial that arises from representation theory, rather than from the categorifications arising from Floer homology. There is ongoing research into the connection between knot Floer homology and quantum $\mathfrak{sl}_n$ homology (see, for instance, \cite{Manolescu2014}), and a quantum $\mathfrak{gl}(1|1)$ homology theory may give a helpful intermediary.

It would also be interesting to relate this to categorifications of $\dot U_q(\mathfrak{gl}(1|1))$ directly. The positive half of $\dot U_q(\mathfrak{gl}(m|1))$ was categorified by Khovanov and Sussan \cite{Khovanov2014}, and certain integral forms of $\dot U_q(\mathfrak{gl}(1|1))$ were categorified by Tian \cite{Tian2012}.

\begin{acknowledgement}
I am grateful to the organisers and speakers of the 2013 SMS at the Universit\'{e} de Montr\'{e}al, where I first learned of some of the ideas used in this paper, particularly in talks by Aaron Lauda. I am also grateful to the referees of this paper for helping to clarify the exposition. This work is supported by an EPSRC doctoral training grant.
\end{acknowledgement}

\section{MOY calculus for the Alexander polynomial}
\subsection{MOY moves}\label{se:moymoves}
It is possible to specialise MOY calculus (see \cite{Murakami1998}) to the case $N=0$, to give the following moves:

\begin{figure}[h!]\begin{equation}\label{zero}\tag{Move 0}
\left(\begin{tikzpicture}[baseline=-0.65ex]
\draw (0,0) circle (0.5);
\draw (0.7,0) node {$i$};
\end{tikzpicture}\right) = 0
\end{equation}
\begin{equation}\label{one}\tag{Move 1}
\left(\begin{tikzpicture}[baseline=-0.65ex,scale=0.6]
\draw (0,-1.5) -- (0,1.5);
\draw (0,-0.5) to [out=270, in=270] (1,-0.2) to [out=90, in=270] (1,0.2) to [out=90, in=90] (0,0.5);
\draw (-0.25, -1) node {$i$};
\draw (-0.25, 1) node {$i$};
\draw (-0.7, 0) node {$j+i$};
\draw (1.25,0) node {$j$};
\end{tikzpicture}\right) = (-1)^j\chuse{i+j-1}{j} 
\left(\begin{tikzpicture}[baseline=-0.65ex,scale=0.6]
\draw (0,-1.5) -- (0,1.5);
\draw (-0.25,0) node {$i$};
\draw (0.25,0) node {};
\end{tikzpicture}\right)
\end{equation}
\begin{equation}\label{two}\tag{Move 2}
\left(\begin{tikzpicture}[baseline=-0.65ex,scale=0.6]
\draw (0.5,-1.5) -- (0.5,-0.6);
\draw (0.5, -0.6) to [out=90, in=270] (0,0) to [out=90,in=270] (0.5,0.6);
\draw (0.5, -0.6) to [out=90, in=270] (1,0) to [out=90,in=270] (0.5,0.6);
\draw (0.5,0.6) -- (0.5,1.5);
\draw (0.75, -1) node {$i$};
\draw (0.75, 1.1) node {$i$};
\draw (-0.7, 0) node {$i-j$};
\draw (1.25, 0) node {$j$};
\end{tikzpicture}\right) = \chuse{i}{j}
\left(\begin{tikzpicture}[baseline=-0.65ex,scale=0.6]
\draw (0.5,-1.5) -- (0.5, 1.5);
\draw (0.75, 0) node {$i$};
\draw (0.25, 0) node {};
\end{tikzpicture}\right)
\end{equation}
\begin{equation}\label{three}\tag{Move 3}
\left(\begin{tikzpicture}[baseline=-0.65ex,scale=0.6]
\draw (0,-1.5) -- (0,-0.5);
\draw (0,-0.5) to [out=90,in=270] (-1,0.5) to [out=90,in=270] (-2,1.5);
\draw (-1,0.5) to [out=90, in=270] (0,1.5);
\draw (0,-0.5) to [out=90,in=270] (1,0.5) to [out=90,in=270] (1,1.5);
\draw (1.3,-1) node {$i+j+k$};
\draw (1.25, 1.25) node {$k$};
\draw (-1.5,0) node {$i+j$};
\draw (-2.25, 1.25) node {$i$};
\draw (-0.5,1.25) node {$j$};
\end{tikzpicture}\right)=
\left(\begin{tikzpicture}[baseline=-0.65ex,scale=0.6]
\draw (0,-1.5) -- (0,-0.5);
\draw (0,-0.5) to [out=90,in=270] (-1,0.5) to [out=90,in=270] (-1,1.5);
\draw (1,0.5) to [out=90, in=270] (2,1.5);
\draw (0,-0.5) to [out=90,in=270] (1,0.5) to [out=90,in=270] (0,1.5);
\draw (1.3,-1.25) node {$i+j+k$};
\draw (2.25, 1.25) node {$k$};
\draw (1.5,0) node {$j+k$};
\draw (-1.25, 1.25) node {$i$};
\draw (0.5,1.25) node {$j$};
\end{tikzpicture}\right)
\end{equation}
\begin{equation}\label{four}\tag{Move 4}
\left(\begin{tikzpicture}[baseline=-0.65ex]
\draw (0,-1) to [out=90,in=180] (0.5,-0.5) to [out=0,in=180] (1,-0.5) to [out=0, in=90] (1.5,-1);
\draw (0.5,-0.5) to [out=180, in=270] (0,0) to [out=90,in=180] (0.5,0.5);
\draw (1,-0.5) to [out=0, in=270] (1.5,0) to [out=90,in=0] (1,0.5);
\draw (0,1) to [out=270,in=180] (0.5,0.5) to [out=0,in=180] (1,0.5) to [out=0, in=270] (1.5,1);
\draw (-0.25, -0.75) node {$1$};
\draw (1.75, -0.75) node {$i$};
\draw (-0.25, 0) node {$i$};
\draw (-0.25, 0.75) node {$1$};
\draw (1.75, 0.75) node {$i$};
\draw (1.75, 0) node {$1$};
\draw (0.75,-0.75) node {$i+1$};
\draw (0.75,0.75) node {$i+1$};
\end{tikzpicture}\right)
= -[i+1]
\left(\begin{tikzpicture}[baseline=-0.65ex]
\draw (0,-1) to [out=90, in=180] (0.5,-0.5) to [out=0,in=180] (1,-0.5) to [out=0,in=90] (1.5,-1);
\draw (0,1) to [out=270, in=180] (0.5,0.5) to [out=0,in=180] (1,0.5) to [out=0,in=270] (1.5,1);
\draw (0.75,-0.5) to [out=180, in=270] (0.25,0) to [out=90,in=180] (0.75,0.5);
\draw (-0.25, -0.75) node {$1$};
\draw (1.75, -0.75) node {$i$};
\draw (1, 0) node {$i-1$};
\draw (-0.25, 0.75) node {$1$};
\draw (1.75, 0.75) node {$i$};
\end{tikzpicture}\right)
+
\left(\begin{tikzpicture}[baseline=-0.65ex]
\draw (0,-1) -- (0,1);
\draw (1,-1) -- (1,1);
\draw (-0.25,0) node {$1$};
\draw (1.25, 0) node {$i$};
\end{tikzpicture}\right)
\end{equation}
\begin{equation}\label{five}\tag{Move 5}
\left(\begin{tikzpicture}[baseline=-0.65ex]
\draw (0,-1)-- (0,1);
\draw (2,-1)-- (2,1);
\draw (0,0.66) to [out=270,in=180] (0.5,0.5) to [out=0,in=180] (1.5,0.5) to [out=0, in=90] (2,0.33);
\draw (0,-0.66) to [out=90, in=180] (0.5,-0.5) to [out=0,in=180] (1.5,-0.5) to [out=0, in=270] (2,-0.33);
\draw (0.1,-1.25) node {$k$};
\draw (2,-1.25) node {$l$};
\draw (0.45,0) node {$k-s$};
\draw (0.1,1.25) node {$k-s+r$};
\draw (1, -0.75) node {$s$};
\draw (1.6,0) node {$l+s$};
\draw (1, 0.66) node {$r$};
\draw (2,1.25) node {$l+s-r$};
\end{tikzpicture}\right)
= \sum_t \chuse{k-l+r-s}{t}\left(\begin{tikzpicture}[baseline=-0.65ex]
\draw (0,-1)-- (0,1);
\draw (2,-1)-- (2,1);
\draw (0,-0.33) to [out=270,in=180] (0.5,-0.5) to [out=0,in=180] (1.5,-0.5) to [out=0, in=90] (2,-0.66);
\draw (0,0.33) to [out=90, in=180] (0.5,0.5) to [out=0,in=180] (1.5,0.5) to [out=0, in=270] (2,0.66);
\draw (0.1,-1.25) node {$k$};
\draw (2,-1.25) node {$l$};
\draw (0.7,0) node {$k+r-t$};
\draw (0.1,1.25) node {$k+r-s$};
\draw (1, -0.75) node {$r-t$};
\draw (2.7,0) node {$l-r+t$};
\draw (1, 0.66) node {$s-t$};
\draw (2,1.25) node {$l-r+s$};
\end{tikzpicture}\right)
\end{equation}
\caption{The six MOY moves}\label{MOYmoves}\end{figure}

At a crossing we either have exactly two strands oriented inwards, or exactly two strands oriented outwards. Choosing an orientation for one of the strands on the bottom of the diagram uniquely determines the orientation for the rest of the diagram, where the curved lines at a vertex indicate when two strands share the same orientation.
 
Defining crossings by figure \ref{fi:knot} and evaluating diagrams using the above MOY moves gives a link invariant (just as described originally in \cite{Murakami1998}, with the parameter $n=0$). In fact, it is enough to specialise to colours in $\{1,2\}$, by the same argument as \cite[Proposition 5.2]{Grant2013}.   However, of course, since all closed diagrams reduce to circles after applying a sequence of MOY moves, this link invariant will be identically $0$.

\begin{figure}[h]\[
\left(\begin{tikzpicture}[baseline=-0.65ex]
\draw[->] (-1,-1) -- (1,1);
\draw (1,-1) -- (0.25, -0.25);
\draw[<-] (-1,1) -- (-0.25,0.25);
\end{tikzpicture}\right) = q\left(\begin{tikzpicture}[baseline=-0.65ex]
\draw[->] (-0.75,-1) -- (-0.75,1);
\draw[->] (0.75,-1) -- (0.75,1);
\draw (-1,0.5) node {$1$};
\draw (1,0.5) node {$1$};
\end{tikzpicture}\right)
- \left( \begin{tikzpicture}[baseline=-0.65ex]
\draw[->] (-0.75,-1) to [out=90,in=270] (0,-0.33) to [out=90,in=270] (0,0.33) to [out=90,in=270] (-0.75,1);
\draw (0.75,-1) to [out=90,in=270] (0,-0.33);
\draw[->] (0,0.33) to [out=90,in=270] (0.75,1);
\draw (-1, -0.8) node {$1$};
\draw (-1,0.8) node {$1$};
\draw (1, -0.8) node {$1$};
\draw (1, 0.8) node {$1$};
\draw (0.25,0) node {$2$};
\end{tikzpicture}\right)
\]

\[
\left(\begin{tikzpicture}[baseline=-0.65ex]
\draw (-1,-1) -- (-0.25,-0.25);
\draw[->] (0.25,0.25) -- (1,1);
\draw[<-] (-1,1) -- (1,-1);
\end{tikzpicture}\right) = q^{-1}\left(\begin{tikzpicture}[baseline=-0.65ex]
\draw[->] (-0.75,-1) -- (-0.75,1);
\draw[->] (0.75,-1) -- (0.75,1);
\draw (-1,0.5) node {$1$};
\draw (1,0.5) node {$1$};
\end{tikzpicture}\right)
- \left( \begin{tikzpicture}[baseline=-0.65ex]
\draw[->] (-0.75,-1) to [out=90,in=270] (0,-0.33) to [out=90,in=270] (0,0.33) to [out=90,in=270] (-0.75,1);
\draw (0.75,-1) to [out=90,in=270] (0,-0.33);
\draw[->] (0,0.33) to [out=90,in=270] (0.75,1);
\draw (-1, -0.8) node {$1$};
\draw (-1,0.8) node {$1$};
\draw (1, -0.8) node {$1$};
\draw (1, 0.8) node {$1$};
\draw (0.25,0) node {$2$};
\end{tikzpicture}\right)
\]\caption{MOY resolutions of knot diagrams}\label{fi:knot}\end{figure}

As we will see in section \ref{se:relationtodiagrams}, this is actually a feature of the Reshetikhin-Turaev procedure applied to $\mathfrak{gl}(1|1)$. The resolution of this problem is to choose a basepoint on the link, and cut the link at that point to yield a $(1,1)$-tangle. The idea of cutting basepoints to obtain link invariants is well-established, and appears for example in \cite{Sartori2013} and explored in a more general context in \cite{Geer2009}. Then we can apply the MOY moves until we have a polynomial times a single strand, and define $\Delta(L)$ to be this polynomial. Then $\Delta(L)$ is equal to the Alexander polynomial of $L$. Observe that the Alexander polynomial of the unknot is $1$, we have $\Delta(L_+)-\Delta(L_-)=(q-q^{-1})\Delta(L_0)$ (where as usual $L_+$, $L_-$ and $L_0$ indicate a link differing only in one crossing that is positive, negative and resolved respectively), and also the Alexander polynomial of any split link is $0$.

\subsection{The category of MOY diagrams}
We can define a category of trivalent graphs $\operatorname{Tri}$ as the category with objects given by finite sequences of natural numbers, and morphisms generated by
\begin{align*}
\begin{tikzpicture}[baseline=-0.65ex]
\draw (-0.5,-0.5) -- (0,0);
\draw (0,0) -- (0.5,-0.5);
\draw[->] (0,0) -- (0,0.5);
\draw (-0.5,-0.75) node {$k$};
\draw (0.5,-0.75) node {$l$};
\draw (0,0.75) node {$k+l$};
\end{tikzpicture} &&  \begin{tikzpicture}[baseline=-0.65ex]
\draw[<-] (-0.5,0.5) -- (0,0);
\draw[->] (0,0) -- (0.5,0.5);
\draw (0,0) -- (0,-0.5);
\draw (-0.5,0.75) node {$k$};
\draw (0.5,0.75) node {$l$};
\draw (0,-0.75) node {$k+l$};
\end{tikzpicture} \\ 
\begin{tikzpicture}[baseline=-0.65ex]
\draw[<-] (-0.5,0.5) [out=270,in=180] to (0,0) [out=0,in=270] to (0.5,0.5);
\end{tikzpicture} && \begin{tikzpicture}[baseline=-0.65ex]
\draw[->] (-0.5,0.5) [out=270,in=180] to (0,0) [out=0,in=270] to (0.5,0.5);
\end{tikzpicture} \\
\begin{tikzpicture}[baseline=-0.65ex]
\draw[<-] (-0.5,-0.5) [out=90,in=180] to (0,0) [out=0,in=90] to (0.5,-0.5);
\end{tikzpicture} && \begin{tikzpicture}[baseline=-0.65ex]
\draw[->] (-0.5,-0.5) [out=90,in=180] to (0,0) [out=0,in=90] to (0.5,-0.5);
\end{tikzpicture}
\end{align*}
with the cups and caps of each colour in $\N$. Morphisms are read from bottom to top, and composition is given by placing diagrams on top of one another, and set to $0$ if the labellings on the identified edges do not match. We can then define the category $\operatorname{MOY}$ to be the quotient of this category by the MOY relations in figure \ref{MOYmoves}. The category $\operatorname{MOY}$ is monoidal, with the monoidal structure given by placing diagrams next to one another horizontally.

\subsection{Categorification problems}
In their paper \cite{Khovanov2004}, Khovanov and Rozansky categorify the polynomial associated to $U_q(\mathfrak{sl}(n))$ by producing chain group summands $C(\Gamma)$ associated to a resolution of a link diagram that obey categorified versions of the MOY moves. A similar approach is taken in \cite{Wu2009} with the coloured versions. Ideally we would like to be able to do the same thing here, but MOY moves for the Alexander polynomial have minus signs that do not appear in the $\mathfrak{sl}(n)$ case, which suggests any categorification takes a different flavour to the categorifications of $\mathfrak{sl}(n)$ polynomials, as the minus signs suggest the use of chain complexes in a way that was not necessary in those cases. However, we observe that the only time the moves \ref{one} and \ref{four} appear in Reidemeister moves is when there are strands pointing downwards: therefore if we restrict to braid diagrams and braid-like Reidemeister moves, it becomes possible to categorify the remaining MOY moves. A categorification of our category $\operatorname{Rep}$ of $U_q(\mathfrak{gl}(1|1))$-modules was obtained in \cite{Sartori2013a}, using methods from the BGG category $\mathcal{O}$. However, the categorification of the non-MOY relation (see Section \ref{se:additional}) is still conjectural in that setting (see \cite[Conjecture 7.7]{Sartori2013a}).

In \cite{Gilmore2010}, Gilmore shows invariance of the knot Floer cube of resolutions by observing that it satisfies some of these MOY relations, but also seems to require non-local relations. 

\section{Representation theory of \texorpdfstring{$U_q(\mathfrak{gl}(1|1))$}{Uq(gl(1|1))}}
In this section we give an overview of the quantum group $U_q(\mathfrak{gl}(1|1)$ and its representation theory, and how this relates to the Alexander polynomial. Our main reference for this section is \cite{Sartori2013}, although we use different notation and different conventions in places in order to more closely tie in with the sequel.
\subsection{The quantum group}
We let $U_q(\mathfrak{gl}(1|1))$ be generated as a $\Z/2\Z$-graded $\C(q)$-algebra by $E,F$ in degree $1$ and $L_1^{\pm 1},L_2^{\pm 1}$ in degree $0$ with relations
\[ L_1L_2=L_2L_1,\quad L_1 E=q EL_1, \quad L_2 E =q^{-1} E L_2 \]
\[ L_1 F = q^{-1} F L_1,\quad L_2 F = q F L_2 \]
\[ EF + FE = \frac{K-K^{-1}}{q-q^{-1}},\quad K=L_1L_2 \]
\[ E^2=F^2=0. \]
It is well-known it is possible to define a Hopf superalgebra structure on $U_q(\mathfrak{gl}(1|1))$. We choose comultiplication to be
\[ \Delta (E)= E\otimes K^{-1} + 1\otimes E,\quad  \Delta (F)=F\otimes 1 + K\otimes F \]
\[ \Delta(L_1)=L_1\otimes L_1,\quad  \Delta(L_2)=L_2\otimes L_2,\quad \Delta(K)=K\otimes K \]
and the antipode to be
\[ S(E)=-EK, S(F)=-K^{-1}F, S(L_1)=L_1^{-1}, S(L_2)=L_2^{-1} \]
with counit
\[ \epsilon (E)=\epsilon(F)=0, \epsilon (L_1)=\epsilon(L_2)=1. \]

Due to the Hopf superalgebra structure, given a representation $V$ we can define the dual representation $V^*$ using the antipode. Also, given two representations $V,W$, the comultiplication gives the action of $U_q(\mathfrak{gl}(1|1))$ on their tensor product $V\otimes W$.

\subsection{Representation Theory}\label{se:representationtheory}
The standard representation $\C_q^{1|1}$ of $U_q(\mathfrak{gl}(1|1))$ is generated as a vector space by $v,w$ with
\begin{align*}
 Ev=0,&& Fv=w,&& L_1 v = qv,&& L_2v=v \\
 Ew=v,&& Fw=0,&& L_1 w = w,&& L_2w = qw.
\end{align*}
Here $v$ is in degree $0$ and $w$ is in degree $1$. Since the quantum dimension of $\C^{1|1}_q$ is given by the supertrace of the action of $K=L_1L_2$, we see that $\dim_q \C^{1|1}_q=0$.

For this representation, the action of the $R$-matrix is given by
\[ R:\C^{1|1}_q\otimes \C^{1|1}_q \to \C^{1|1}_q\otimes \C^{1|1}_q:\left\lbrace \begin{aligned}
w\otimes w &\mapsto -q^{-1}w\otimes w \\
v\otimes w &\mapsto (q-q^{-1})v\otimes w + w\otimes v\\
w\otimes v &\mapsto v\otimes w \\
v\otimes v &\mapsto qv\otimes v
\end{aligned}\right. \]
\[ R^{-1}:\C^{1|1}_q\otimes \C^{1|1}_q \to \C^{1|1}_q\otimes \C^{1|1}_q: \left\lbrace \begin{aligned}
w\otimes w &\mapsto -q w\otimes w \\
v\otimes w &\mapsto w\otimes v\\
w\otimes v &\mapsto (q^{-1}-q)w\otimes v + v\otimes w\\
v\otimes v &\mapsto q^{-1}v\otimes v
\end{aligned}\right. \]

Following \cite{Berenstein2008}, we define the braided exterior algebra\footnote{The term `braided' here refers to the fact that it depends on the braided monoidal structure of the module category.} of a representation $V$ as
\[ \bigwedge_q V= TV/(S^2_q(V)) \]
where $TV=\bigoplus_n V^{\otimes n}$, the symmetric square $S^2_q(V)$ of $TV$ is the vector subspace generated by all elements $x\in V\otimes V$ such that $R(x)=q^r x$ for some $r\in \Z$ (ie. all eigenvectors of $R$ with positive eigenvalues), and $(S^2_q(V))$ is the two-sided ideal generated by $S^2_q(V)$. This definition of the symmetric square ensures that the action of $U_q(\mathfrak{gl}(1|1))$ on $V^{\otimes n}$ descends to an action of $U_q(\mathfrak{gl}(1|1))$ on $\bigwedge^n_q(V)$. Then
\[ \bigwedge_q \C_q^{1|1}= T\C^{1|1}_q/(v\otimes v , qv\otimes w+w\otimes v). \]
Note that $\bigwedge^i_q \C^{1|1}_q$ is non-trivial for all $i\geq 0$, since $w\wedge \cdots \wedge w$ is non-zero. As a vector space, $\bigwedge \C^{1|1}_q \cong \bigwedge \langle v\rangle \otimes S \langle w \rangle$, where $S$ is the symmetric algebra. For all $i$, $\bigwedge^i \C^{1|1}$ is $2$-dimensional (as a vector space), and is spanned by $v\wedge w \wedge \cdots \wedge w$ and $w\wedge w \wedge \cdots \wedge w$, which we denote $v\wedge w^{i-1}$ and $w^i$ respectively.

Note that the dual representation $(\C^{1|1}_q)^*$ does not occur as $\bigwedge^i_q(\C^{1|1}_q)$ for any $i$, in contrast to the case of the dual of the standard representation $\C^m_q$ over $U_q(\mathfrak{gl}(m))$. Note also that each of the representations $\bigwedge^i_q(\C^{1|1}_q)$ is simple.

\subsection{Important Maps}
The product in the braided exterior algebra $\bigwedge_q(\C^{1|1}_q)$ is given by the wedge product:
\[ M_{k,l}:\bigwedge^k_q \C^{1|1}_q \otimes \bigwedge_q ^l \C^{1|1}_q \to \bigwedge_q^{k+l} \C^{1|1}_q : \left\lbrace\begin{aligned} w^k\otimes w^l &\mapsto  w^{k+l}\\ v\wedge w^{k-1} \otimes w^l &\mapsto  v\wedge w^{k+l-1} \\
w^k \otimes v\wedge w^{l-1} &\mapsto (-q)^{k}  v\wedge w^{k+l-1}\\
v\wedge w^{k-1}\otimes v\wedge w^{l-1}&\mapsto 0
\end{aligned}\right.\]
It also has a coproduct, given as follows:
\[ M'_{k,l}:\left\lbrace\begin{aligned}\bigwedge^{k+l}_q \C^{1|1}_q &\to \bigwedge^k_q \C^{1|1}_q\otimes \bigwedge^l_q \C^{1|1}_q
\\w^{k+l}&\mapsto \chuse{k+l}{k} w^k\otimes w^l\\
 v\wedge w^{k+l-1} &\mapsto 
 q^{-l}\chuse{k+l-1}{l}v\wedge w^{k-1}\otimes w^l + (-1)^k \chuse{k+l-1}{k}w^k\otimes v\wedge w^{l-1}
\end{aligned}\right.\]
We also have maps
\[ \operatorname{coev}: \C(q) \to \bigwedge^i_q \C^{1|1}_q\otimes (\bigwedge^i_q\C^{1|1}_q)^*: 1 \mapsto v\wedge w^{i-1}\otimes (v\wedge w^{i-1})^* + w^i\otimes (w^i)^* \]
\[ \widehat{\operatorname{coev}}:\C(q) \to (\bigwedge^i_q\C^{1|1}_q)^*\otimes \bigwedge^i_q\C^{1|1}_q: 1\mapsto (-1)^{i-1}(q^{-i} (v\wedge w^{i-1})^*\otimes v\wedge w^{i-1} -  q^{-i} (w^i)^*\otimes w^i)   \]
\[ \operatorname{ev}:(\bigwedge^i_q\C^{1|1}_q)^*\otimes \bigwedge^i_q\C^{1|1}_q \to \C(q): \begin{aligned} (v\wedge w^{i-1})^*\otimes v\wedge w^{i-1} &\mapsto 1 \\ (w^i)^*\otimes w^i &\mapsto 1
\end{aligned}\]
\[\widehat{\operatorname{ev}}: \bigwedge^i_q\C^{1|1}_q\otimes (\bigwedge^i_q\C^{1|1}_q)^* \to \C(q): \begin{aligned}
v\wedge w^{i-1}\otimes (v\wedge w^{i-1})^* &\mapsto (-1)^{i-1}q^i \\ w^i\otimes (w^i)^* &\mapsto (-q)^i.
\end{aligned}\]

These arise from the general theory of ribbon Hopf algebras. The ribbon element $v$ acts as the identity on $\C^{1|1}_q$ and the element $u$ (with $S^2(x)=uxu^{-1}$ for any $x$) acts as $K$ (see, for example, \cite{Sartori2013} or \cite{Viro2007} for more details).

\subsection{Relation to diagram calculus}\label{se:relationtodiagrams}
There is a functor from the category of trivalent graphs $\operatorname{Tri}$ to the category of representations of $U_q(\mathfrak{gl}(1|1))$, defined on generating morphisms as follows:
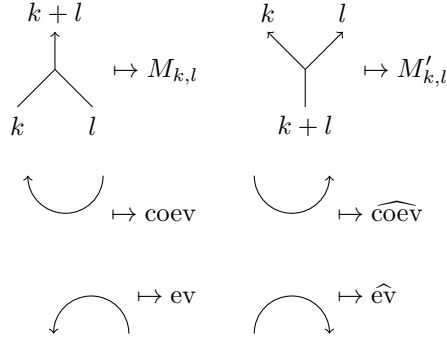
\[ \begin{tikzpicture}[baseline=-0.65ex]
\draw (-0.5,-0.5) -- (0,0);
\draw (0,0) -- (0.5,-0.5);
\draw[->] (0,0) -- (0,0.5);
\draw (-0.5,-0.75) node {$k$};
\draw (0.5,-0.75) node {$l$};
\draw (0,0.75) node {$k+l$};
\end{tikzpicture} \mapsto M_{k,l}
\quad \quad
\begin{tikzpicture}[baseline=-0.65ex]
\draw[<-] (-0.5,0.5) -- (0,0);
\draw[->] (0,0) -- (0.5,0.5);
\draw (0,0) -- (0,-0.5);
\draw (-0.5,0.75) node {$k$};
\draw (0.5,0.75) node {$l$};
\draw (0,-0.75) node {$k+l$};
\end{tikzpicture} \mapsto M'_{k,l}
\]

\[\begin{tikzpicture}[baseline=-0.65ex]
\draw[<-] (-0.5,0.5) [out=270,in=180] to (0,0) [out=0,in=270] to (0.5,0.5);
\end{tikzpicture} \mapsto \operatorname{coev}
\quad\quad
\begin{tikzpicture}[baseline=-0.65ex]
\draw[->] (-0.5,0.5) [out=270,in=180] to (0,0) [out=0,in=270] to (0.5,0.5);
\end{tikzpicture} \mapsto \widehat{\operatorname{coev}}
\]

\[\begin{tikzpicture}[baseline=-0.65ex]
\draw[<-] (-0.5,-0.5) [out=90,in=180] to (0,0) [out=0,in=90] to (0.5,-0.5);
\end{tikzpicture} \mapsto \operatorname{ev}
\quad\quad
\begin{tikzpicture}[baseline=-0.65ex]
\draw[->] (-0.5,-0.5) [out=90,in=180] to (0,0) [out=0,in=90] to (0.5,-0.5);
\end{tikzpicture} \mapsto \widehat{\operatorname{ev}}
\]

\begin{prop}
This functor obeys the MOY relations, hence descends to a monoidal functor from the category $\operatorname{MOY}$ of MOY diagrams to the category of representations.
\end{prop}
We prove this by checking each relation as follows:
\subsubsection*{Move 0}
Move 0 is obtained by either $\operatorname{ev}\circ\, \widehat{\operatorname{coev}}$ or $\widehat{\operatorname{ev}}\circ \operatorname{coev}$, which can be seen in either case to give $0$.

\subsubsection*{Move 1}
We check the case $i=k,j=1$ by following the map from bottom to top:
\begin{align*}
v\wedge w^{k-1} &\mapsto v\wedge w^{k-1}\otimes (v\otimes v^*+w\otimes w^*)\\ &\mapsto v\wedge w^{k} \otimes w^* \\ &\mapsto  q^{-1}[k]v\wedge w^{k-1}\otimes w \otimes w^* + (-1)^k w^{k}\otimes v\otimes w^* \\ & \mapsto -[k]v\wedge w^{k-1}. \\ w^k &\mapsto w\otimes (v\otimes v^*+ w\otimes w^*) \\ & \mapsto (-q)^k v\wedge w^k \otimes v^*+w^{k+1} \otimes w^*\\  &\mapsto (-q)^kq^{-1}[k]v\wedge w^{k-1}\otimes w \otimes v^* + q^k w^k \otimes v\otimes v^* +[k+1]w^k\otimes w \otimes w^* \\ &\mapsto q^{k+1} w^k-q[k+1]w^k=-[k]w^k \end{align*}
so the map acts as $-[k]\id$ on basis elements as required. The case for general $j$ can be checked similarly. Alternatively, one could define similar maps to $M_{k,l}$ and $M'_{k,l}$ on the dual $(\C_q^{1|1})^*$ such that one can slide trivalent vertices around cups and caps. The result would then follow by induction on $j$ by using Move 2 to split the strand labelled $j$ into two strands labelled $j-1$ and $1$, and using Move 3 to rearrange the diagram into two nested applications of Move 1. The mirror image of this move can be checked similarly, applying $\widehat{\operatorname{coev}}$ and $\operatorname{ev}$.

\subsubsection*{Move 2}
We verify $M_{k,l}\circ M'_{k,l}=\chuse{k+l}{l}$. This follows from the identity
\[ q^{-l}\chuse{k+l-1}{l}+q^{k}\chuse{k+l-1}{l-1} = \chuse{k+l}{l}. \]

\subsubsection*{Move 3}
This follows from coassociativity of the comultiplication, and with arrows reversed follows from associativity of multiplication.

\subsubsection*{Move 4}
Since the diagrams are not in Morse position, to interpret this map we must first apply an isotopy so that the strands labelled $i+1$ are now at a diagonal slant and oriented upwards. Then the map for the left-hand diagram acts on $w\otimes (w^{i-1})^*$ as:
\begin{align*}
w\otimes (w^i)^* &\mapsto (-1)^{i-1}(q^{-i}(v\wedge w^{i-1})^*\otimes v\wedge w^{i-1} - q^{-i} (w^i)^*\otimes w^i)\otimes w\otimes (w^i)^* \\
& \mapsto (-1)^{i-1} (q^{-i}(v\wedge w^{i-1})^*\otimes v\wedge w^i\otimes (w^i)^* - q^{-i} (w^i)^*\otimes w^{i+1}\otimes (w^i)^*) \\
& \mapsto (-1)^{i-1}(q^{-i}(v\wedge w^{i-1})^*\otimes (q^{-i}v\otimes w^i - [i]w\otimes v\wedge w^{i-1})\otimes (w^i)^*\\
& \quad\quad - q^{-i}[i+1] (w^i)^*\otimes w \otimes w^{i}\otimes (w^i)^*) \\
&\mapsto -q^{-i} (v\wedge w^{i-1})^*\otimes v +[i+1](w^i)^*\otimes w\\
& \mapsto (-q^{-i} (v\wedge w^{i-1})^*\otimes v +[i+1](w^i)^*\otimes w)\otimes (v\wedge w^{i-1}\otimes (v\wedge w^{i-1})^*+ w^i\otimes (w^i)^*) \\
& \mapsto -q^{-i}(v\wedge w^{i-1})^*\otimes v\wedge w^i \otimes (w^i)^* -q[i+1](w^i)^*\otimes v\wedge w^i \otimes (v\wedge w^{i-1})^* \\ & \quad\quad + [i+1](w^i)^*\otimes w^{i+1}\otimes (w^{i})^* \\
&\mapsto  -q^{-i}(v\wedge w^{i-1})^*\otimes (q^{-1}[i]v\wedge w^{i-1}\otimes w +(-1)^i w^i\otimes v)\otimes (w^i)^* \\
& \quad\quad -q[i+1](w^i)^*\otimes (q^{-1}[i]v\wedge w^{i-1}\otimes w +(-1)^i w^i\otimes v)\otimes (v\wedge w^{i-1})^* \\
& \quad\quad +[i+1]^2 (w^{i})^*\otimes w^i\otimes w \otimes (w^i)^* \\
&\mapsto -q^{-i-1}[i]w\otimes (w^i)^* - (-1)^iq[i+1]v\otimes (v\wedge w^{i-1})^* + [i+1]^2w\otimes (w^{i})^*
\end{align*}
while the left-hand diagram on the right-hand side acts as
\begin{align*}
w\otimes (w^i)^* &\mapsto (-1)^{i-2}(q^{-i+1}(v\wedge w^{i-2})^*\otimes v\wedge w^{i-2} - q^{-i+1} (w^{i-1})^*\otimes w^{i-1})\otimes w\otimes (w^{i})^*\\
&\mapsto (-1)^{i-2}(q^{-i+1}(v\wedge w^{i-2})^*\otimes v\wedge w^{i-1}\otimes (w^i)^* - q^{-i+1}(w^{i-1})^*\otimes w^i\otimes (w^i)^*) \\
&\mapsto -q (w^{i-1})^*\\
&\mapsto -q (w^{i-1})^* \otimes (v\wedge w^{i-1}\otimes (v\wedge w^{i-1})^* + w^i\otimes (w^i)^*) \\
&\mapsto -q (w^{i-1})^*\otimes (q^{-1}[i-1]v\wedge w^{i-2}\otimes w + (-1)^{i-1}w^{i-1}\otimes v) \otimes (v\wedge w^{i-1})^*\\ &\quad\quad - q[i] (w^{i-1})^*\otimes w^{i-1}\otimes (w^i)^*\\
&\mapsto (-1)^i q v\otimes (v\wedge w^{i-1})^* - q[i]w\otimes (w^{i})^*.
\end{align*}
So by comparing coefficients and using $q[i][i+1]+1=[i+1]^2-q^{-i-1}[i]$, we get the result for this basis element, and the others are similar.

\subsubsection*{Move 5}
It suffices to verify for $r=s=1$, since if $r,s$ are larger, then we can split the edges of the left-hand diagram labelled $r,s$ using move 2 into two pairs of edges with one of the pair labelled $1$, then slide the trivalent vertices using move 3 to obtain the left-hand side of move 5 with only $s=r=1$. The resulting diagrams after applying move 5 have colour at most $s-1$, $r-1$, so the result follows by induction. The diagram on the left-hand side acts on basis elements as:
\begin{align*} w^k\otimes w^l &\mapsto [k][l+1] w^{k}\otimes w^{l} \\
v\wedge w^{k-1}\otimes w^l &\mapsto (q^{-1}[k-1][l+1]+q^{k-l-1})v\wedge w^{k-1}\otimes w^l + (-1)^k[l]w^k\otimes v\wedge w^{l-1} \\
w^k\otimes v\wedge w^{l-1} &\mapsto (-1)^k q^{k-l}[k]v\wedge w^{k-1}\otimes w^l + q[k][l]w^k\otimes v\wedge w^{l-1} \\
v\wedge w^{k-1}\otimes v\wedge w^{l-1} &\mapsto [k-1][l]v\wedge w^{k-1}\otimes v\wedge w^{l-1}.
\end{align*}
Similarly, the $t=0$ diagram on the right-hand side acts as
\begin{align*} w^k\otimes w^l &\mapsto [k+1][l] w^{k}\otimes w^{l} \\
v\wedge w^{k-1}\otimes w^l &\mapsto q^{-1}[k][l]v\wedge w^{k-1}\otimes w^l + (-1)^k[l]w^k\otimes v\wedge w^{l-1} \\
w^k\otimes v\wedge w^{l-1} &\mapsto (-1)^k q^{k-l}[k]v\wedge w^{k-1}\otimes w^l + (q^{k-l+1}+q[k+1][l-1])w^k\otimes v\wedge w^{l-1} \\
v\wedge w^{k-1}\otimes v\wedge w^{l-1} &\mapsto [k][l-1]v\wedge w^{k-1}\otimes v\wedge w^{l-1}
\end{align*}
and clearly the $t=1$ diagram is a multiple of the identity morphism, so the result follows from the identity
\[ [a][b+1] = [b][a+1]+[a-b] \]
for all $a,b\in\Z$.

\subsection{The link invariant}
To construct a link invariant from this, one can take a braid diagram whose closure is the link and interpret this as a map $(\C^{1|1}_q)^{\otimes n} \to (\C^{1|1}_q)^{\otimes n}$. Taking the quantum supertrace of $n-1$ of these tensor factors yields a map $\C^{1|1}_q \to \C_q^{1|1}$ which is equal to $\Delta(L)\cdot\id_{\C^{1|1}_q}$ where $\Delta(L)\in \C(q)$ (see \cite{Geer2009} for more discussion on the role of such traces). In fact, one finds that $\Delta(L)$ is a Laurent polynomial and is independent of the choice of tensor factors involved in the trace, and the braid presentation of $L$ (this is essentially \cite[Theorem~4.6]{Sartori2013}).

\subsection{Non-MOY diagram identity}\label{se:additional}
The maps in the previous subsection satisfy an additional identity that does not arise from MOY calculus:
\begin{lem} The functor defined before satisfies the relation:
\begin{equation}\label{eqn:nonmoy} \chuse{k}{t}\chuse{l}{s}\begin{tikzpicture}[scale=0.3,baseline=0.7cm]
\draw (0,0) -- (0,1);
\draw (0,1) to [out=90,in=270] (2,3);
\draw (2,3) to [out=270,in=90] (4,1);
\draw (4,1) -- (4,0);
\draw (2,3) -- (2,4);
\draw (2,4) to [out=90,in=270] (1,5);
\draw (2,4) to [out=90,in=270] (3,5);
\draw (0,1) -- (0,6);
\draw (4,1) -- (4,6);
\draw (1,5) -- (1,6);
\draw (3,5) -- (3,6);
\coordinate [label=left:$k$] (k) at (0,0.5);
\coordinate [label=right:$l$] (l) at (4,0.5);
\coordinate [label=below:$t$] (t) at (1,2);
\coordinate [label=below:$s$] (s) at (3,2);
\coordinate [label=below:$t$] (tt) at (1,5);
\coordinate [label=below:$s$] (ss) at (3,5);
\end{tikzpicture}
-\chuse{l}{s}
\begin{tikzpicture}[scale=0.3,baseline=0.7cm]
\draw (0,0) -- (0,1);
\draw (0,1) to [out=90,in=270] (2,3);
\draw (2,3) to [out=270,in=90] (4,1);
\draw (4,1) -- (4,0);
\draw (2,3) -- (2,4);
\draw (2,4) to [out=90,in=270] (0,6);
\draw (2,4) to [out=90,in=270] (3,5);
\draw (0,1) -- (0,8);
\draw (4,1) -- (4,8);
\draw (3,5) -- (3,8);
\draw (0,7) to [out=90,in=270] (1,8);
\coordinate [label=left:$k$] (k) at (0,0.5);
\coordinate [label=right:$l$] (l) at (4,0.5);
\coordinate [label=below:$t$] (t) at (1,2);
\coordinate [label=below:$s$] (s) at (3,2);
\coordinate [label=below:$t$] (tt) at (1,5);
\coordinate [label=below:$s$] (ss) at (3,5);
\coordinate [label=below:$t$] (ttt) at (1,8);
\end{tikzpicture}
- \chuse{k}{t}
\begin{tikzpicture}[scale=0.3,baseline=0.7cm]
\draw (0,0) -- (0,1);
\draw (0,1) to [out=90,in=270] (2,3);
\draw (2,3) to [out=270,in=90] (4,1);
\draw (4,1) -- (4,0);
\draw (2,3) -- (2,4);
\draw (2,4) to [out=90,in=270] (1,5);
\draw (2,4) to [out=90,in=270] (4,6);
\draw (0,1) -- (0,8);
\draw (4,1) -- (4,8);
\draw (1,5) -- (1,8);
\draw (4,7) to [out=90,in=270] (3,8);
\coordinate [label=left:$k$] (k) at (0,0.5);
\coordinate [label=right:$l$] (l) at (4,0.5);
\coordinate [label=below:$t$] (t) at (1,2);
\coordinate [label=below:$s$] (s) at (3,2);
\coordinate [label=below:$t$] (tt) at (1,5);
\coordinate [label=below:$s$] (ss) at (3,5);
\coordinate [label=below:$s$] (sss) at (3,8);
\end{tikzpicture}
+
\begin{tikzpicture}[scale=0.3,baseline=0.7cm]
\draw (0,0) -- (0,1);
\draw (0,1) to [out=90,in=270] (2,3);
\draw (2,3) to [out=270,in=90] (4,1);
\draw (4,1) -- (4,0);
\draw (2,3) -- (2,4);
\draw (2,4) to [out=90,in=270] (0,6);
\draw (2,4) to [out=90,in=270] (4,6);
\draw (0,1) -- (0,8);
\draw (4,1) -- (4,8);
\draw (4,7) to [out=90,in=270] (3,8);
\draw (0,7) to [out=90,in=270] (1,8);
\coordinate [label=left:$k$] (k) at (0,0.5);
\coordinate [label=right:$l$] (l) at (4,0.5);
\coordinate [label=below:$t$] (t) at (1,2);
\coordinate [label=below:$s$] (s) at (3,2);
\coordinate [label=below:$t$] (tt) at (1,5);
\coordinate [label=below:$s$] (ss) at (3,5);
\coordinate [label=below:$t$] (ttt) at (1,8);
\coordinate [label=below:$s$] (sss) at (3,8);
\end{tikzpicture}=0\end{equation}
for $k,l\geq 2$ and $t,s\geq 1$. 
\end{lem}
This relation is partly redundant, as we will only require $t=s=1$ to prove Theorem \ref{th:faithful}, meaning that the more general case must follow from the $t=s=1$ case.
\begin{proof}
Each term in the sum sends
\[ w^k\otimes w^l \mapsto \chuse{k}{t}^2\chuse{l}{s}^2\chuse{s+t}{t} w^{k-t}\otimes w^t\otimes w^s \otimes w^{l-s} \]
so all terms cancel for this basis element. In order, each diagram acts on $v\wedge w^{k-1}\otimes w^l$ as follows:
\begin{align*}
v\wedge w^{k-1}\otimes w^l \mapsto & q^{-t}\chuse{k-1}{t}\chuse{l}{s}\chuse{t+s}{s}v\wedge w^{k-t-1}\otimes w^t\otimes w^s\otimes w^{l-s}\\ &+ (-1)^{k-t}q^{-s}\chuse{k-1}{t}\chuse{l}{s}\chuse{t+s-1}{s} w^{k-t}\otimes v\wedge w^{t-1}\otimes w^s\otimes w^{l-s}\\ &+ (-1)^k \chuse{k-1}{k-t}\chuse{l}{s}\chuse{t+s-1}{t} w^{k-t}\otimes w^t\otimes v\wedge w^{s-1}\otimes w^{l-s}\end{align*}

\begin{align*}
v\wedge w^{k-1}\otimes w^l \mapsto & q^{-2t}\chuse{k-1}{t}^2\chuse{l}{s}\chuse{t+s}{s}v\wedge w^{k-t-1}\otimes w^t\otimes w^s\otimes w^{l-s}\\ 
&+ (-1)^{k-t}q^{-t}\chuse{k-1}{t}\chuse{k-1}{k-t}\chuse{l}{s}\chuse{t+s-1}{s} w^{k-t}\otimes v\wedge w^{t-1}\otimes w^s\otimes w^{l-s}\\ 
&+ q^{k-2t-s}\chuse{k-1}{t}\chuse{k-1}{k-t}\chuse{l}{s}\chuse{t+s-1}{s}v\wedge w^{k-t-1}\otimes w^t\otimes w^s\otimes w^{l-s} \\
&+ (-1)^{k-t}q^{k-t-s}\chuse{k-1}{k-t}^2\chuse{l}{s}\chuse{t+s-1}{s}w^{k-t}\otimes v\wedge w^{t-1}\otimes w^s\otimes w^{l-s} \\
&+ (-1)^k \chuse{k-1}{k-t}\chuse{k}{t}\chuse{l}{s}\chuse{t+s-1}{t} w^{k-t}\otimes w^t\otimes v\wedge w^{s-1}\otimes w^{l-s}
\end{align*}

\begin{align*}
v\wedge w^{k-1}\otimes w^l \mapsto & q^{-t}\chuse{k-1}{t}\chuse{l}{s}^2\chuse{t+s}{s}v\wedge w^{k-t-1}\otimes w^t\otimes w^s\otimes w^{l-s}\\ &+ (-1)^{k-t}q^{-s}\chuse{k-1}{t}\chuse{l}{s}^2\chuse{t+s-1}{s} w^{k-t}\otimes v\wedge w^{t-1}\otimes w^s\otimes w^{l-s}\\ &+ (-1)^k q^{-l+s} \chuse{k-1}{k-t}\chuse{l}{s}\chuse{l-1}{l-s}\chuse{t+s-1}{t} w^{k-t}\otimes w^t\otimes v\wedge w^{s-1}\otimes w^{l-s}\\
&+ (-1)^{k+s}\chuse{k-1}{k-t}\chuse{l}{s}\chuse{l-1}{s}\chuse{t+s-1}{t} w^{k-t}\otimes w^t \otimes w^s\otimes v\wedge w^{l-s-1}
\end{align*}

\begin{align*}
v\wedge w^{k-1}\otimes w^l \mapsto & q^{-2t}\chuse{k-1}{t}^2\chuse{l}{s}^2\chuse{t+s}{s}v\wedge w^{k-t-1}\otimes w^t\otimes w^s\otimes w^{l-s}\\ 
&+ (-1)^{k-t}q^{-t}\chuse{k-1}{t}\chuse{k-1}{k-t}\chuse{l}{s}^2\chuse{t+s-1}{s} w^{k-t}\otimes v\wedge w^{t-1}\otimes w^s\otimes w^{l-s}\\ 
&+ q^{k-2t-s}\chuse{k-1}{t}\chuse{k-1}{k-t}\chuse{l}{s}^2\chuse{t+s-1}{s}v\wedge w^{k-t-1}\otimes w^t\otimes w^s\otimes w^{l-s} \\
&+ (-1)^{k-t}q^{k-t-s}\chuse{k-1}{k-t}^2\chuse{l}{s}^2\chuse{t+s-1}{s}w^{k-t}\otimes v\wedge w^{t-1}\otimes w^s\otimes w^{l-s} \\
&+ (-1)^kq^{-l+s} \chuse{k-1}{k-t}\chuse{k}{t}\chuse{l}{s}\chuse{l-1}{l-s}\chuse{t+s-1}{t} w^{k-t}\otimes w^t\otimes v\wedge w^{s-1}\otimes w^{l-s}\\
&+ (-1)^{k+s} \chuse{k-1}{k-t}\chuse{k}{t}\chuse{l}{s}\chuse{l-1}{s}\chuse{t+s-1}{t}w^{k-t}\otimes w^t \otimes w^s\otimes v\wedge w^{l-s-1}
\end{align*}
Careful comparison of coefficients, when coefficients of diagrams are included, shows that these all cancel in the sum, hence the sum acts as $v\wedge w^{k-1}\otimes w^l \mapsto 0$. The vectors $w^k\otimes w^l$ and $v\wedge w^{k-1}\otimes w^l$ generate the representation $\bigwedge_q^k\C^{1|1}_q\otimes \bigwedge_q^l\C^{1|1}_q$, so the result follows.
\end{proof}

\begin{remark}
This relation is not necessary for evaluating closed MOY diagrams, or for evaluating diagrams arising from links with a component cut open at a basepoint, since the MOY relations were already sufficient to do this. However, the above shows that the relation is consistent with the MOY relations.
\end{remark}
\begin{remark} In another form, this relation also appears in \cite[Definition~5.5]{Sartori2013a} for the case $k=l=2$, and arose in that context from idempotents projecting onto simple representations of the Hecke algebra. This seems to be a quantisation of a relation well-known to experts, although it is hard to find a source.
\end{remark}

\section{Quantum Skew Howe Duality}
We prove an important result that relates to commuting actions of the quantum groups $U_q(\mathfrak{gl}(m))$ and $U_q(\mathfrak{gl}(1|1))$ on certain modules, which allows us to prove Corollary \ref{co:full} that gives a full functor from idempotented versions of $U_q(\mathfrak{gl}(m))$ to a category of representations of $U_q(\mathfrak{gl}(1|1))$. This will be crucial in what follows.

Much of this section follows the proof of the corresponding result obtained in \cite{Cautis2012}. More information about braided exterior algebras can be found in \cite{Berenstein2008}. The algebras $\dot{U}_q(\mathfrak{gl}(m))/I_{\lambda}$ used in the proof of Theorem \ref{th:faithful} are generalised q-Schur algebras introduced by Doty \cite{Doty2003}.

\subsection{The Quantum Group \texorpdfstring{$U_q(\mathfrak{gl}(m))$}{Uq(gl(m))}}\label{se:glm}
The quantum group $U_q(\mathfrak{gl}(m))$ is the unital $\C(q)$-algebra generated by $E_i,F_i$ for $1\leq i \leq m-1$, and $K_i^{\pm 1}$ for $1\leq i\leq m$ subject to the relations
\[K_iK_i^{-1}=1,\quad K_iK_j=K_jK_i \]
\[ K_i E_i = q E_i K_i,\quad  K_{i+1} E_{i} = q^{-1} E_{i} K_{i+1},\quad K_i E_j = E_j K_i \quad \mathrm{if}\quad j\neq i,i-1 \]
\[ K_i F_i = q^{-1} F_i K_i,\quad K_{i+1} F_i = q F_i K_{i+1} ,\quad K_i F_j = F_j K_i \quad \mathrm{if}\quad j\neq i,i-1\]
\[ E_iF_j-F_jE_i=\delta_{ij} \frac{K_iK_{i+1}^{-1}-K_i^{-1}K_{i+1}}{q-q^{-1}} \]
\[ E_i^2E_j-(q+q^{-1})E_iE_jE_i+E_jE_i^2=0 \quad \mathrm{if}\quad j=i\pm 1 \]
\[ F_i^2F_j-(q+q^{-1})F_iF_jF_i+F_jF_i^2=0 \quad \mathrm{if}\quad j=i\pm 1 \]
\[ E_iE_j=E_jE_i, \quad F_iF_j=F_jF_i \quad \mathrm{if}\quad |i-j|>1. \]

As before, we can choose a coproduct
\[ \Delta(E_i)=E_i\otimes K_iK_{i+1}^{-1} + 1\otimes E_i, \quad \Delta(F_i)=F_i\otimes 1 + K_i^{-1}K_{i+1}\otimes F_i, \quad \Delta(K_i)=K_i\otimes K_i \]
which along with the antipode
\[ S(K_i)=K_i^{-1}, \quad S(E_i)=-E_iK_i^{-1}K_{i+1}, \quad S(F_i)=-K_iK_{i+1}^{-1}F_i \]
and counit
\[ \epsilon(K_i)=1, \quad \epsilon(E_i)=0, \quad \epsilon(F_i)=0 \]
makes $U_q(\mathfrak{gl}(m))$ into a Hopf algebra.

The weight lattice of $U_q(\mathfrak{gl}(m))$ is $\Z^m$, and an element is called a weight. Recall that a weight $\lambda$ is said to be dominant if the entries of $\lambda$ form a partition of some natural number. There is a partial order on the weight lattice where  $\lambda\geq \mu$ if and only if $\lambda-\mu=\sum_i c_i\alpha_i$ where $\alpha_i=(0,\ldots,0,1,-1,0,\ldots,0)$ with $1$ in the $i$th position, and $c_i\geq 0$ for all $i$. The Weyl group acts on the weight lattice by permuting the entries.

\subsection{Quantum Skew Howe Duality}
In this subsection we will mostly be concerned with the module
\[ \bigwedge_q(\C_q^{1|1}\otimes \C_q^m) \]
over $U_q(\mathfrak{gl}(1|1)\oplus \mathfrak{gl}(m))=U_q(\mathfrak{gl}(1|1))\otimes U_q(\mathfrak{gl}(m))$ (note that any Lie algebra is trivially a Lie superalgebra with the degree $1$ subspace equal to $0$). This is defined using the $R$-matrix as in  \ref{se:representationtheory} following \cite{Berenstein2008}, where $R$ acts only as permutation of factors on $\C^{m}_q\otimes \C^{1|1}_q$ and $\C^{1|1}_q\otimes \C^m_q$, and as the $R$-matrix of $U_q(\mathfrak{gl}(m))$ on $\C^m_q\otimes \C^m_q$ and as the $R$-matrix of $U_q(\mathfrak{gl}(1|1))$ on $\C^{1|1}_q\otimes \C^{1|1}_q$. So it follows that
\[ R_{\C^{1|1}_q\otimes \C^m_q}=\tau_{23}\circ (R_{\C^{1|1}_q}\otimes R_{\C^m_q})\circ \tau_{23} \]
where $\tau_{23}$ is the map that permutes the middle of the four tensor factors, and hence the braided symmetric square has the form
\begin{equation}\label{symmetricsquare} S_q^2(\C^{1|1}_q\otimes \C^m_q)=\tau_{23}\left((S^2_q(\C^{1|1}_q)\otimes S^2_q(\C^m_q) )\oplus(\bigwedge^2_q(\C^{1|1}_q)\otimes \bigwedge^2_q(\C^m_q))\right).\end{equation}

\begin{lem}\label{le:flat}
The $q=1$ specialisation of $\bigwedge_q(\C_q^{1|1}\otimes \C_q^m)$ is isomorphic to $\bigwedge(\C^{1|1}\otimes \C^m)$ as a module over $U(\mathfrak{gl}(m)\oplus \mathfrak{gl}(1|1))$.
\end{lem}
\begin{proof}
We follow \cite{Berenstein2008}. Let $\{x_i\mid 1\leq i\leq m\}$ be a weight basis for $\C^m_q$ and define $v_i=v\otimes x_i$ and $w_i=w\otimes x_i$. Then $\C^{1|1}_q\otimes \C^m_q$ has a basis $\{v_i,w_i\mid 1\leq i\leq m\}$. Using equation (\ref{symmetricsquare}), we see that $S_q^2(\C^{1|1}_q\otimes \C^m_q)$ is freely spanned by elements of the form
\[ v_i\otimes v_i \]
\[ v_i\otimes v_j+qv_j\otimes v_i \]
\[ qv_i\otimes w_i+w_i\otimes v_i \]
\[ qw_i\otimes w_j-w_j\otimes w_i \]
\[ qv_i\otimes w_j + w_i\otimes v_j +q^2v_j\otimes w_i+qw_j\otimes v_i \]
\[ qv_i\otimes w_j - q^2w_i\otimes v_j - v_j\otimes w_i + q w_j\otimes v_i \]
where $1\leq i < j \leq m$ (or $1\leq i \leq m$ for expressions not involving $j$). Setting $a=v_i,b=v_j,c=w_i,d=w_j$ with $i<j$ yields
\begin{equation}\label{eq:relations}
a^2=b^2=0, ab=-qba, ca=-qac, dc=qcd, cb=-bc, ad+da=(q-q^{-1})bc. \end{equation}
Then $\bigwedge_q(\C^{1|1}_q\otimes \C^m_q)$ is generated by $\{v_i,w_i\mid 1\leq i\leq m\}$ subject to the relations \ref{eq:relations}. Thus there is a spanning set of $\bigwedge^n_q(\C^{1|1}_q\otimes \C^m_q)$ given by elements of the form $v_{i_1}\wedge \cdots \wedge v_{i_l}\wedge w_{i_{l+1}}\wedge\cdots\wedge w_{i_n}$ with $1\leq i_1 <\cdots<i_l\leq m$ and $1\leq i_{l+1}\leq\ldots\leq i_n\leq m$, which is linearly independent since the specialisation to $q=1$ is linearly independent. So this basis specialises at $q=1$ to a basis of $\bigwedge(\C^{1|1}\otimes \C^m)$, so this implies the graded dimension of this $q=1$ specialisation is equal to the graded dimension of $\bigwedge(\C^{1|1}\otimes \C^m)$ and the result follows.
\end{proof}
The above presentation of $\bigwedge_q(\C_q^{1|1}\otimes \C_q^m)$ in terms of $v_i$ and $w_i$ makes it easy to see the actions of $U_q(\mathfrak{gl}(m))$ and $U_q(\mathfrak{gl}(1|1))$. The action of $U_q(\mathfrak{gl}(m))$ is
\begin{align*}
E_k (v_i)= \delta_{i-1,k}v_{i-1} && F_k(v_i)=\delta_{ik}v_{i+1}&& K_k(v_i)=q^{\delta_{i,k}}v_i \end{align*}
and the same for $w_i$, and the action of $U_q(\mathfrak{gl}(1|1))$ is
\begin{align*}
E(v_i)=0 && F(v_i)=w_i && L_1(v_i)=qv_i && L_2(v_i)=v_i \\
\\ E(w_i)=v_i && F(w_i)=0 && L_1(w_i)=w_i && L_2(w_i)=qw_i. \end{align*}
\begin{thm}\label{th:skewhoweduality}
The actions of $U_q(\mathfrak{gl}(m))$ and $U_q(\mathfrak{gl}(1|1))$ on $\bigwedge_q(\C_q^{1|1}\otimes \C_q^m)$ generate each others commutant. As $U_q(\mathfrak{gl}(1|1))$ representations, there is an isomorphism
\[ \bigwedge_q(\C_q^{1|1}\otimes \C_q^m)\cong \left(\bigwedge_q \C_q^{1|1}\right)^{\otimes m} \]
and the $(k_1,\ldots,k_m)$-weight space for the action of $U_q(\mathfrak{gl}(m))$ is identified with
\[ \bigwedge^{k_1}_q\C^{1|1}_q \otimes \cdots \otimes \bigwedge^{k_m}_q\C^{1|1}_q. \]
\end{thm}
\begin{proof}
Let $V_m(\mu)$ denote the irreducible representation of $U(\mathfrak{gl}(m))$ of highest weight $\mu$, taken to be $0$ if $\mu$ has more than $m$ parts. Similarly, let $V_{(1|1)}(\mu)$ be the irreducible representation of $U_q(\mathfrak{gl}(1|1))$ of highest weight $\mu$, taken to be $0$ unless $\mu_2\leq 1$.

By \cite[Theorem~3.3]{Cheng2001}, we have
\begin{equation}\label{eq:weightdecomposition} \bigwedge(\C^{1|1}\otimes \C^m)\cong \bigoplus_{\lambda\in H} V_{(1|1)}(\lambda^t)\otimes V_m(\lambda) \end{equation}
where $H$ is the set of hook-shaped partitions, that is, partitions satisfying $\lambda_2\leq 1$, and $\lambda^t$ is the reflection of the Young diagram along the diagonal. So it is enough to show that this same decomposition into highest-weight modules holds in the quantised case.

Since it is semi-simple, the algebra $\bigwedge_q(\C_q^{1|1}\otimes \C_q^m)$ decomposes as a direct sum of irreducible modules over $U_q(\mathfrak{gl}(m)\oplus \mathfrak{gl}(1|1))$, which are of the form $V\otimes W$ where $V$ is an irreducible module of $U_q(\mathfrak{gl}(1|1))$ and $W$ is an irreducible module of $U_q(\mathfrak{gl}(m))$. So
\[ \bigwedge_q(\C_q^{1|1}\otimes \C_q^m) \cong \bigoplus_i V_i\otimes W_i \]
where $V_i$ and $W_i$ are highest-weight modules. To see the highest-weights occurring in this direct sum decomposition, it is enough to know the dimensions of the weight-spaces.

But by Lemma \ref{le:flat}, we know that $\bigwedge_q(\C_q^{1|1}\otimes \C_q^m)$ specialises at $q=1$ to $\bigwedge(\C^{1|1}\otimes \C^m)$, and so the dimensions of weight-spaces is the same for both modules. Therefore we have the direct sum
\[ \bigwedge_q(\C_q^{1|1}\otimes \C_q^m) \cong \bigoplus_{\lambda\in H} V_{(1|1)}(\lambda^t)\otimes V_m(\lambda) \]
of highest-weight modules of $U_q(\mathfrak{gl}(m)\oplus \mathfrak{gl}(1|1))$, as required.

As this sum is multiplicity-free, it follows that the actions of $U_q(\mathfrak{gl}(1|1))$ and $U_q(\mathfrak{gl}(m))$ generate each other's commutants.

The isomorphism
\[ \left(\bigwedge_q \C_q^{1|1}\right)^{\otimes m} \to \bigwedge_q(\C_q^{1|1}\otimes \C_q^m) \]
is constructed using the maps
\[ \phi_i:\bigwedge_q\C_q^{1|1}\to \bigwedge_q(\C_q^{1|1}\otimes \C_q^m): v\mapsto v_i, w\mapsto w_i \]
by setting $\phi=\phi_1\wedge \phi_2\wedge\cdots \wedge\phi_m$. The identification of the weight spaces then follows.
\end{proof}

\begin{lem}\label{le:actionofglm}
The action of $F_i^{(r)}\in U_q(\mathfrak{gl}(m))$ on $\bigwedge^{k_1}_q\C^{1|1}_q \otimes \cdots \otimes \bigwedge^{k_m}_q\C^{1|1}_q$ is given by \[\id \otimes \cdots \otimes M_{r,k_{i+1}}\circ M'_{k_i-r,r} \otimes \cdots \otimes \id\] and the action of $E_i^{(r)}$ is
\[ \id \otimes \cdots \otimes M_{k_i,r}\circ M'_{r,k_{i+1}-r} \otimes \cdots \otimes \id. \]
\end{lem}
\begin{proof}
It is enough to show this for $m=2$. Let $\langle x_+,x_- \rangle$ be the standard representation of $U_q(\mathfrak{gl}(2))$ with $F(x_+)=x_-$, $E(x_-)=x_+$ and $K(x_+)=qx_+$, $K(x_-)=q^{-1}x_-$. Using the coproduct on $U_q(\mathfrak{gl}(2))$ with the basis $w_+,w_-,v_+,v_-$ of $\bigwedge_q(\C^{1|1}_q\otimes \C^2_q)$, we have
\[ F(w_+\wedge w_+)=q^{-1}w_+\wedge w_-+w_-\wedge w_+ = (q+q^{-1})w_+\wedge w_- \]
which is exactly $M'_{1,1}(w\wedge w)$. Similarly,
\[ F(v_+\wedge w_+)=q^{-1}v_+\wedge w_-+v_-\wedge w_+ = q^{-1}v_+\wedge w_- - w_+\wedge v_- \]
so the action of $F$ on the $(2,0)$ weight space is exactly the comultiplication $M'_{1,1}$.

Now we check the $(1,1)$ weight space. We have
\[ F(w_+\wedge w_-)=w_-\wedge w_- \]
\[ F(v_+\wedge w_-)=v_-\wedge w_- \]
\[ F(w_+\wedge v_-)= w_-\wedge v_-=-qv_-\wedge w_- \]
which again is exactly the action of $M_{(1,1)}$ on $\C^{1|1}_q\otimes \C^{1|1}_q$.

Now by induction on $k$ we check the action of $F$ on the $(k,0)$ weight-space. For convenience we write $w_+^k$ for $w_+\wedge w_+\wedge \cdots\wedge w_+$ and similarly for other terms. Then
\begin{align*}
F(w_+^k)&=K^{-1}w_+\wedge F(w_+^{k-1}) + F(w_+)\wedge w_+^{k-1}\\ &= q^{-1}[k-1]w_+^{k-1}\wedge w_- + w_-\wedge w_+^{k-1}\\ &= (q^{-1}[k-1]+q^{k-1})w_+^{k-1}\wedge w_-\\ &= [k] w_+^{k-1}\wedge w_-
\end{align*}
\[ F(v_+\wedge w_+^{k-1})= q^{-1}[k]v_+\wedge w_+^{k-2}\wedge w_- + (-1)^k w_+^{k-1}\wedge v_- \]
which agrees with the comultiplication.

Now by induction on $l$ we check the action of $F^{(l)}$ on $(k+l,0)$. We have
\begin{align*}
F^{(l)}(w_+^{k+l}) &= \frac{1}{[l]} F(F^{(l-1)}(w_+^{k+l}))\\
&= \frac{1}{[l]} F(\chuse{k+l}{l-1} w_+^{k+1}\wedge w_-^{l-1}) \\
&= \frac{[k+1]}{[l]}\chuse{k+l}{l-1}w_+^k \wedge w_-^l \\
&= \chuse{k+l}{l} w_+^k \wedge w_-^l
\end{align*}
\[ F^{(l)}(v_+\wedge w_+^{k+l-1}) = q^{-l}\chuse{k+l-1}{l} v_+\wedge w^{k-1}_+\wedge w^{l}_- + (-1)^k \chuse{k+l-1}{k} w_+^k \wedge v_-\wedge w_-^{l-1} \]
which is $M'_{k,l}$ as wanted.

Finally, we can check $F^{(l)}$ on the $(a,b)$ weight space by
\[ F^{(l)}(w_+^a\wedge w_-^b) = \chuse{a}{l}w_+^{a-l}\wedge w_-^{b+l} \]
\[ F^{(l)}(w_+^a\wedge v_-\wedge w_-^{b-1})= (-q)^l \chuse{a}{l} w_+^{a-l}\wedge v_-\wedge w_-^{b+l-1} \]
\[ F^{(l)}(v_+\wedge w_+^{a-1}\wedge w_-^b) = q^{-l}\chuse{a-1}{l}v_+\wedge w_+^{a-l-1}\wedge w_-^{b+l} +(-1)^{a-l} \chuse{a-1}{l-1}w_+^{a-l}\wedge v_-\wedge w_-^{b+l-1} \]
which is exactly the effect of $M_{l,b}\circ M'_{a-l,l}$ as wanted. The proof for $E$ is similar.
\end{proof}

Lemma \ref{le:actionofglm} suggests that there is a diagrammatic interpretation of the action of $U_q(\mathfrak{gl}(m))$, since it acts by the maps appearing in Section \ref{se:relationtodiagrams}. We make this interpretation explicit in the following sections.

\subsection{Idempotented Quantum \texorpdfstring{$\mathfrak{gl}(m)$}{gl(m)}}
Following Lusztig \cite{Lusztig1993}, we can define the idempotented version of the quantum group $U_q(\mathfrak{gl}(m))$, which we denote by $\dot{U}_q(\mathfrak{gl}(m))$ as follows. We define the algebra $U'_q(\mathfrak{gl}(m))$ by adjoining the additional generators $1_{\lambda}$ to $U_q(\mathfrak{gl}(m))$ for each $\lambda\in \Z^m$ with the additional relations
\[ 1_\lambda 1_{\lambda'}=\delta_{\lambda \lambda'}1_\lambda \]
\[ E_i 1_\lambda = 1_{\lambda+\alpha_i} E_i \]
\[ F_i 1_\lambda = 1_{\lambda-\alpha_i} F_i \]
\[ K_i 1_\lambda = q^{\lambda_i} 1_\lambda \]
where $\alpha_i=(0,\ldots,1,-1,\ldots,0)$ with the $1$ in position $i$, and $\lambda_i$ is the $i$th term of $\lambda$. Then we define
\[ \dot U_q(\mathfrak{gl}(m))= \bigoplus_{\lambda,\mu\in \Z^m} 1_\mu U'_q(\mathfrak{gl}(m)) 1_\lambda. \]
We will usually think of this as a category with objects $1_\lambda$ for $\lambda\in \Z^m$ and morphisms $1_\lambda\to 1_\nu$ given by $1_\nu \dot{U}_q(\mathfrak{gl}(m))1_\lambda$. Observe that $1_\nu \dot{U}_q(\mathfrak{gl}(m))1_\lambda=0$ unless $\sum_i \lambda_i=\sum_i \nu_i$.

We denote by $\dot{U}_q^\infty (\mathfrak{gl}(m))$ the quotient of the category $\dot{U}_q(\mathfrak{gl}(m))$ by the identity morphisms of the objects $1_\lambda$ with $\lambda_i<0$ for some $i$.

\subsection{Ladders}\label{se:ladders}
Ladders were introduced in \cite{Cautis2012} to give a diagrammatic description of certain subcategories of $\dot{U}_q (\mathfrak{gl}(m))$. We use them here to give a similar description of $\dot{U}_q^\infty (\mathfrak{gl}(m))$.
\begin{definition}A ladder with $m$ uprights is a diagram in $[0,1]\times [0,1]$ with $m$ oriented vertical lines oriented upwards from the bottom edge to the top edge, some number of horizontal rungs between adjacent uprights, and a natural number labelling each rung and segment of uprights such that the algebraic sum of labels at each trivalent vertex is 0.
\end{definition}

\begin{definition}
The category $\operatorname{FLad}_m$ is defined to have objects $(k_1,\ldots,k_m)$ for $k_i\geq 0$, and morphisms $(k_1,\ldots,k_m)\to (l_1,\ldots,l_m)$ given by linear combinations of ladders such that the bottom segments of the uprights have labels $(k_1,\ldots,k_m)$ in order, and the top segments have labels $(l_1,\ldots,l_m)$.
\end{definition}

The $F$ in $\operatorname{FLad}_m$ is for `free', to distinguish with $\operatorname{Lad}_m$ defined below. We can define a functor $\psi_m:\operatorname{FLad}_m\to \dot{U}_q^\infty (\mathfrak{gl}(m))$ by $(k_1,\ldots,k_m)\mapsto 1_{(k_1,\ldots,k_m)}$ on objects, and on morphisms:
\[ \begin{tikzpicture}[baseline=-0.65ex]
\draw (-0.75,-1) -- (-0.75,1);
\draw (0.75,-1) -- (0.75,1);
\draw (-0.75,0.25) -- (0.75,-0.25);
\draw (-2,-1) -- (-2,1);
\draw (2,-1) -- (2,1);
\draw (0,0.3) node {$r$};
\draw (-1,-1.25) node {$k_i$};
\draw (-1,1.25) node {$k_i+r$};
\draw (-1.4,0) node {$\cdots$};
\draw (1.4,0) node {$\cdots$};
\draw (1,-1.25) node {$k_{i+1}$};
\draw (1,1.25) node {$k_{i+1}-r$};
\draw (-2,-1.25) node {$k_1$};
\draw (2,-1.25) node {$k_m$};
\draw (-2,1.25) node {$k_1$};
\draw (2,1.25) node {$k_m$};
\end{tikzpicture} \mapsto E^{(r)}_i 1_{\underline{k}} \]

\[ \begin{tikzpicture}[baseline=-0.65ex]
\draw (-0.75,-1) -- (-0.75,1);
\draw (0.75,-1) -- (0.75,1);
\draw (-0.75,-0.25) -- (0.75,0.25);
\draw (-2,-1) -- (-2,1);
\draw (2,-1) -- (2,1);
\draw (0,0.3) node {$r$};
\draw (-1,-1.25) node {$k_i$};
\draw (-1,1.25) node {$k_i-r$};
\draw (-1.4,0) node {$\cdots$};
\draw (1.4,0) node {$\cdots$};
\draw (1,-1.25) node {$k_{i+1}$};
\draw (1,1.25) node {$k_{i+1}+r$};
\draw (-2,-1.25) node {$k_1$};
\draw (2,-1.25) node {$k_m$};
\draw (-2,1.25) node {$k_1$};
\draw (2,1.25) node {$k_m$};
\end{tikzpicture} \mapsto F^{(r)}_i 1_{\underline{k}} \]
where $F_i^{(r)}=F_i^r/[r]!$ and $E_i^{(r)}=E_i^r/[r]!$. Here the uprights are understood to be oriented upwards, and we have drawn the rungs slanting upwards to indicate their orientation. We will stick to this convention henceforth.

\begin{thm}\label{th:ladders}
The kernel of $\psi_m$ is generated by the following relations:
\[ \begin{tikzpicture}[baseline=-0.65ex]
\draw (-1,-1) -- (-1,1);
\draw (0,-1) -- (0,1);
\draw (1,-1) -- (1,1);
\draw (-1,-0.6) -- (0,-0.2);
\draw (1,0.2) -- (0,0.6);
\draw (-1,-1.25) node {$k_1$};
\draw (0,-1.25) node {$k_2$};
\draw (1,-1.25) node {$k_3$};
\draw (-1.6,0.75) node {$k_1-r$};
\draw (0,1.25) node {$k_2+r+s$};
\draw (1.6,0.75) node {$k_3-s$};
\draw (0.5,0.2) node {$s$};
\draw (-0.5,-0.2) node {$r$};
\end{tikzpicture} =  \begin{tikzpicture}[baseline=-0.65ex]
\draw (-1,-1) -- (-1,1);
\draw (0,-1) -- (0,1);
\draw (1,-1) -- (1,1);
\draw (-1,0.2) -- (0,0.6);
\draw (1,-0.6) -- (0,-0.2);
\draw (-1,-1.25) node {$k_1$};
\draw (0,-1.25) node {$k_2$};
\draw (1,-1.25) node {$k_3$};
\draw (-1.6,0.75) node {$k_1-r$};
\draw (0,1.25) node {$k_2+r+s$};
\draw (1.6,0.75) node {$k_3-s$};
\draw (0.5,-0.2) node {$s$};
\draw (-0.5,0.2) node {$r$};
\end{tikzpicture} \]

\[ \begin{tikzpicture}[baseline=-0.65ex]
\draw (-1,-1) -- (-1,1);
\draw (0,-1) -- (0,1);
\draw (1,-1) -- (1,1);
\draw (-1,-0.2) -- (0,-0.6);
\draw (1,0.6) -- (0,0.2);
\draw (-1,-1.25) node {$k_1$};
\draw (0,-1.25) node {$k_2$};
\draw (1,-1.25) node {$k_3$};
\draw (-1.6,0.75) node {$k_1+r$};
\draw (0,1.25) node {$k_2-r-s$};
\draw (1.6,0.75) node {$k_3+s$};
\draw (0.5,0.2) node {$s$};
\draw (-0.5,-0.2) node {$r$};
\end{tikzpicture} =  \begin{tikzpicture}[baseline=-0.65ex]
\draw (-1,-1) -- (-1,1);
\draw (0,-1) -- (0,1);
\draw (1,-1) -- (1,1);
\draw (-1,0.6) -- (0,0.2);
\draw (1,-0.2) -- (0,-0.6);
\draw (-1,-1.25) node {$k_1$};
\draw (0,-1.25) node {$k_2$};
\draw (1,-1.25) node {$k_3$};
\draw (-1.6,0.75) node {$k_1+r$};
\draw (0,1.25) node {$k_2-r-s$};
\draw (1.6,0.75) node {$k_3+s$};
\draw (0.5,-0.2) node {$s$};
\draw (-0.5,0.2) node {$r$};
\end{tikzpicture} \]

\[ \begin{tikzpicture}[baseline=-0.65ex]
\draw (-0.5,-1) -- (-0.5,1);
\draw (0.5,-1) -- (0.5,1);
\draw (-0.5,-0.6) -- (0.5,-0.2);
\draw (-0.5,0.2) -- (0.5,0.6);
\draw (-0.5,-1.25) node {$k_1$};
\draw (0.5,-1.25) node {$k_2$};
\draw (-1,1.25) node {$k_1-r-s$};
\draw (1,1.25) node {$k_2+r+s$};
\draw (0,-0.6) node {$r$};
\draw (0,0.6) node {$s$};
\end{tikzpicture} = \chuse{r+s}{s}
\begin{tikzpicture}[baseline=-0.65ex]
\draw (-0.5,-1) -- (-0.5,1);
\draw (0.5,-1) -- (0.5,1);
\draw (-0.5,-0.2) -- (0.5,0.2);
\draw (-0.5,-1.25) node {$k_1$};
\draw (0.5,-1.25) node {$k_2$};
\draw (-1,1.25) node {$k_1-r-s$};
\draw (1,1.25) node {$k_2+r+s$};
\draw (0,-0.3) node {$r+s$};
\end{tikzpicture}
 \]
 
\[ \begin{tikzpicture}[baseline=-0.65ex]
\draw (-0.5,-1) -- (-0.5,1);
\draw (0.5,-1) -- (0.5,1);
\draw (-0.5,-0.6) -- (0.5,-0.2);
\draw (0.5,0.2) -- (-0.5,0.6);
\draw (-0.5,-1.25) node {$k_1$};
\draw (0.5,-1.25) node {$k_2$};
\draw (-1,1.25) node {$k_1-s+r$};
\draw (1,1.25) node {$k_2+s-r$};
\draw (0,-0.6) node {$s$};
\draw (0,0.6) node {$r$};
\draw (-1.1,0) node {$k_1-s$};
\draw (1.1,0) node {$k_2+s$};
\end{tikzpicture} = \sum_t \chuse{k_1-k_2+r-s}{t}
\begin{tikzpicture}[baseline=-0.65ex]
\draw (-0.5,-1) -- (-0.5,1);
\draw (0.5,-1) -- (0.5,1);
\draw (-0.5,-0.2) -- (0.5,-0.6);
\draw (0.5,0.6) -- (-0.5,0.2);
\draw (-0.5,-1.25) node {$k_1$};
\draw (0.5,-1.25) node {$k_2$};
\draw (-1,1.25) node {$k_1-r+s$};
\draw (1,1.25) node {$k_2+r-s$};
\draw (0,-0.7) node {$r-t$};
\draw (0,0.7) node {$s-t$};
\draw (-1.4,0) node {$k_1+r-t$};
\draw (1.4,0) node {$k_2-r+t$};
\end{tikzpicture}
\]

\[ \begin{tikzpicture}[baseline=-0.65ex]
\draw (-1,-1.4) -- (-1,1.4);
\draw (0,-1.4) -- (0,1.4);
\draw (1,-1.4) -- (1,1.4);
\draw (-1,-1) -- (0,-0.6);
\draw (-1,-0.2) -- (0,0.2);
\draw (0,0.6) -- (1,1);
\draw (-1,-1.65) node {$k_1$};
\draw (0,-1.65) node {$k_2$};
\draw (1,-1.65) node {$k_3$};
\draw (-0.5,-1) node {$1$};
\draw (-0.5,0.2) node {$1$};
\draw (0.5,1) node {$1$};
\end{tikzpicture}
- [2] 
\begin{tikzpicture}[baseline=-0.65ex]
\draw (-1,-1.4) -- (-1,1.4);
\draw (0,-1.4) -- (0,1.4);
\draw (1,-1.4) -- (1,1.4);
\draw (-1,-1) -- (0,-0.6);
\draw (0,-0.2) -- (1,0.2);
\draw (-1,0.6) -- (0,1);
\draw (-1,-1.65) node {$k_1$};
\draw (0,-1.65) node {$k_2$};
\draw (1,-1.65) node {$k_3$};
\draw (-0.5,-1) node {$1$};
\draw (0.5,0.2) node {$1$};
\draw (-0.5,1) node {$1$};
\end{tikzpicture} + 
\begin{tikzpicture}[baseline=-0.65ex]
\draw (-1,-1.4) -- (-1,1.4);
\draw (0,-1.4) -- (0,1.4);
\draw (1,-1.4) -- (1,1.4);
\draw (0,-1) -- (1,-0.6);
\draw (-1,-0.2) -- (0,0.2);
\draw (-1,0.6) -- (0,1);
\draw (-1,-1.65) node {$k_1$};
\draw (0,-1.65) node {$k_2$};
\draw (1,-1.65) node {$k_3$};
\draw (0.5,-1) node {$1$};
\draw (-0.5,0.2) node {$1$};
\draw (-0.5,1) node {$1$};
\end{tikzpicture} = 0
\]
\[\begin{tikzpicture}[baseline=-0.65ex]
\draw (-0.5,-1) -- (-0.5,1);
\draw (0.5,-1) -- (0.5,1);
\draw (-0.5,-0.6) -- (0.5,-0.6);
\draw (1.5,-1) -- (1.5,1);
\draw (2.5,-1) -- (2.5,1);
\draw (1.5,0.6) -- (2.5,0.6);
\draw (-0.5,-1.25) node {$k_1$};
\draw (0.5,-1.25) node {$k_2$};
\draw (1.5,-1.25) node {$k_3$};
\draw (2.5,-1.25) node {$k_4$};
\draw (0,-0.6) node [label=above:$r$] {};
\draw (2,0.6) node [label=above:$s$] {};
\draw (-1,1.25) node {};
\draw (1,1.25) node {};
\draw (-1.1,0) node {};
\draw (1.1,0) node {};
\draw (1,0) node {$\cdots$};
\end{tikzpicture} =
\begin{tikzpicture}[baseline=-0.65ex]
\draw (-0.5,-1) -- (-0.5,1);
\draw (0.5,-1) -- (0.5,1);
\draw (-0.5,0.6) -- (0.5,0.6);
\draw (1.5,-1) -- (1.5,1);
\draw (2.5,-1) -- (2.5,1);
\draw (1.5,-0.6) -- (2.5,-0.6);
\draw (-0.5,-1.25) node {$k_1$};
\draw (0.5,-1.25) node {$k_2$};
\draw (1.5,-1.25) node {$k_3$};
\draw (2.5,-1.25) node {$k_4$};
\draw (0,0.6) node [label=above:$r$] {};
\draw (2,-0.6) node [label=above:$s$] {};
\draw (1,1.25) node {};
\draw (1.1,0) node {};
\draw (1,0) node {$\cdots$};
\end{tikzpicture}
\]
with either orientation on each of the rungs in the last relation as long as the two $r$-coloured rungs have the same orientation, and similarly for the two $s$-coloured rungs. We also include mirror images of the third and fifth relations, and all relations are considered to have arbitrarily many uprights on each side.
\end{thm}
\begin{proof}
This is essentially \cite[Proposition~5.1.2]{Cautis2012}. The proof involves checking the relations on the divided powers $F^{(r)}$ and $E^{(r)}$ from the relations in Section \ref{se:glm}. For example, the third relation comes from
\[ F_i^{(r)}F_i^{(s)}1_\lambda = \chuse{r+s}{s} F^{(r+s)}1_\lambda. \]
The others are similar.
\end{proof}

\begin{definition}
We define $\operatorname{Lad}_m$ to be the quotient of $\operatorname{FLad}_m$ by the relations in Theorem \ref{th:ladders}.
\end{definition}

\begin{cor}\label{co:lad}
The induced functors
\[ \psi_m: \operatorname{Lad}_m \to \dot{U}_q^\infty (\mathfrak{gl}(m)) \]
are equivalences of categories.
\end{cor}

Note that we can make the non-MOY relation \ref{eqn:nonmoy} into a ladder in $\dot U_q(\mathfrak{gl}(4))$. In addition, by attaching the diagram
\[
\begin{tikzpicture}[scale=0.5]
\draw (0,1) -- (0,3);
\draw (0,3) to [out=90,in=270] (1,4);
\draw (1,4) to [out=270,in=90] (2,3);
\draw (2,2) -- (2,3);
\draw (2,2) to [out=270,in=90] (1,1);
\draw (2,2) to [out=270,in=90] (3,1);
\draw (2,3) to [out=90,in=270] (3,4);
\draw (4,1) -- (4,3);
\draw (4,3) to [out=90,in=270] (3,4);
\draw (3,4) -- (3,5);
\draw (1,4) -- (1,5);
\coordinate[label=left:$k-t$] (kt) at (0,1.5);
\coordinate[label=right:$l-s$] (ls) at (4,1.5);
\coordinate[label=left:$k$] (k) at (1,4.5);
\coordinate[label=right:$l$] (l) at (3,4.5);
\coordinate[label=left:$t+s$] (ts) at (2,2.5);
\end{tikzpicture}
\]
to the top of each of the diagrams in the relation, we can also derive a relation involving diagrams that can be made into ladders in $\dot U_q(\mathfrak{gl}(2))$. When $t=s=1$, we use the MOY moves and simplify the quantum integers to obtain
\begin{equation}\label{eq:nonmoy2}
[k+1][k][l][l-1] \left(\begin{tikzpicture}[baseline=-0.65ex]
\draw (-0.5,-1) -- (-0.5,1);
\draw (0.5,-1) -- (0.5,1);
\coordinate[label=left:$k$] (k) at (-0.5,0.5);
\coordinate[label=right:$l$] (l) at (0.5,0.5);
\end{tikzpicture}\right)
- [2][k+1][l-1] \left(\begin{tikzpicture}[baseline=-0.65ex]
\draw (-0.5,-1) -- (-0.5,1);
\draw (0.5,-1) -- (0.5,1);
\draw (-0.5,-0.2) -- (0.5,-0.6);
\draw (0.5,0.6) -- (-0.5,0.2);
\draw (-0.5,-1.25) node {$k$};
\draw (0.5,-1.25) node {$l$};
\draw (-0.5,1.25) node {$k$};
\draw (0.5,1.25) node {$l$};
\draw (0,-0.7) node {$1$};
\draw (0,0.7) node {$1$};
\draw (-1.2,0) node {$k+1$};
\draw (1.2,0) node {$l-1$};
\end{tikzpicture}\right)
+ [2]^2 \left(\begin{tikzpicture}[baseline=-0.65ex]
\draw (-0.5,-1) -- (-0.5,1);
\draw (0.5,-1) -- (0.5,1);
\draw (-0.5,-0.2) -- (0.5,-0.6);
\draw (0.5,0.6) -- (-0.5,0.2);
\draw (-0.5,-1.25) node {$k$};
\draw (0.5,-1.25) node {$l$};
\draw (-0.5,1.25) node {$k$};
\draw (0.5,1.25) node {$l$};
\draw (0,-0.7) node {$2$};
\draw (0,0.7) node {$2$};
\draw (-1.2,0) node {$k+2$};
\draw (1.2,0) node {$l-2$};
\end{tikzpicture}\right)
\end{equation}
as a derived relation.

\begin{definition} We define $\operatorname{Lad}^\Xi_m$ to be the quotient of $\operatorname{Lad}_m$ by relations derived from the relation \ref{eqn:nonmoy} perhaps by attaching additional uprights with no rungs between them to either side.\end{definition}

\begin{definition} We define $\dot{U}_q^\Xi (\mathfrak{gl}(m))$ to be the quotient of $\dot{U}_q^\infty (\mathfrak{gl}(m))$ such that the induced functor
\[ \psi_m: \operatorname{Lad}^\Xi_m \to \dot{U}_q^\Xi (\mathfrak{gl}(m))\]
is an equivalence of categories.
\end{definition}

We denote by $\Xi$ the $2$-sided ideal in $\dot{U}_q^\infty (\mathfrak{gl}(m))$ generated by the image of the relation in Section \ref{se:additional} under the equivalence $\psi_m$.

\subsection{Relationship between idempotented quantum \texorpdfstring{$\mathfrak{gl}(m)$}{gl(m)} and representations of \texorpdfstring{$\mathfrak{gl}(1|1)$}{gl(1|1)}}
Let $\operatorname{Rep}$ be the additive category of $U_q(\mathfrak{gl}(1|1))$-modules monoidally generated by
\[ \bigwedge^{k_1}_q\C^{1|1}_q \otimes \cdots \otimes \bigwedge^{k_s}_q\C^{1|1}_q \]
where $s\in \N$ and $k_i\geq 0$ for all $i$.

For each $m$ we can define a functor
\[ \phi_m: \dot{U}_q(\mathfrak{gl}(m)) \to \operatorname{Rep} \]
by sending $1_\lambda$ to $\bigwedge^{\lambda_1}_q\C^{1|1}_q \otimes \cdots \otimes \bigwedge^{\lambda_m}_q\C^{1|1}_q$ if $\lambda_i\geq 0$ for all $i$, or to $0$ else. The map on morphisms is given by
\[ 1_\nu \dot{U}_q(\mathfrak{gl}(m))1_\lambda \to \operatorname{Hom}_{U_q(\mathfrak{gl}(1|1))}\left( \bigwedge^{\lambda_1}_q\C^{1|1}_q \otimes \cdots \otimes \bigwedge^{\lambda_m}_q\C^{1|1}_q, \bigwedge^{\nu_1}_q\C^{1|1}_q \otimes \cdots \otimes \bigwedge^{\nu_m}_q\C^{1|1}_q\right) \]
which is defined using Theorem \ref{th:skewhoweduality}, since the action of $U_q(\mathfrak{gl}(1|1))$ on the weight-spaces commutes with the action of $\mathfrak{gl}(1|1)$. The fact that the actions generate each other's commutants implies that $\phi_m$ is full.

From the definition of $\dot U_q^\infty (\mathfrak{gl}(m))$, we immediately get the following:
\begin{cor}\label{co:full} The induced functor
\[ \phi:\bigoplus_m \dot{U}_q^\infty (\mathfrak{gl}(m)) \to \operatorname{Rep} \]
is full.
\end{cor}

Clearly this functor cannot be faithful, because of the relation in Section \ref{se:additional} which cannot be derived from any relation in $\dot{U}_q(\mathfrak{gl}(m))$. However, we will show that this relation generates the entire kernel of the functor. 

We say a weight $\mu$ is dominated by $\lambda$ if it lies in the Weyl group orbit of a dominant weight $\mu'$ such that $\lambda\geq \mu'$. A weight $\mu$ dominated by $\lambda$ satisfies $|\mu|=|\lambda|$. In particular, a weight $\mu$ dominated by $(K,0,\ldots,0)$ satisfies $|\mu|=K$. We will consider the algebra
\[ \bigoplus_{K\in\N}\dot{U}_q(\mathfrak{gl}(m))/I_{(K,0,\ldots,0)} \]
where $I_{\lambda}$ is the two-sided ideal of $\dot{U}_q(\mathfrak{gl}(m))$ generated by the objects $1_\mu$ corresponding to those weights $\mu$ that are not dominated by $\lambda$. 

\begin{lem}\label{algiso}
Let $\lambda$ be a dominant weight, and let $L(\lambda)$ be the set of all dominant weights dominated by $\lambda$. Then there is an isomorphism of algebras
\[ \dot{U}_q(\mathfrak{gl}(m))/I_{\lambda} \to \bigoplus_{l\in L(\lambda)}\operatorname{End}_{\C(q)}(V_m(l)). \]
\end{lem}
\begin{proof}
This is Lemma 4.4.2 in \cite{Cautis2012}.
\end{proof}

\begin{thm}\label{th:faithful}
There is an induced functor
\[ \phi_m:\dot{U}^\Xi_q(\mathfrak{gl}(m)) \to \operatorname{Rep} \]
which is fully faithful.
\end{thm}
\begin{proof}
By Lemma \ref{algiso}, we have an isomorphism
\[ \dot{U}_q(\mathfrak{gl}(m))/I_{(K,0,\ldots,0)} \to \bigoplus_{l\in L((K,0,\ldots,0))}\operatorname{End}_{\C(q)}(V_m(l)). \] 
But by Theorem \ref{th:skewhoweduality},
\[ \operatorname{Hom}_{U_q(\mathfrak{gl}(1|1))}\left(\bigwedge_q (\C^{1|1}_q \otimes \C^m_q),\bigwedge_q (\C^{1|1}_q \otimes \C^m_q)\right) = \bigoplus_{\mu\in H}\operatorname{End}_{\C(q)}(V_m(\mu))) \]
where $H$ is the set of dominant weights $\mu$ such that $\mu_2\leq 1$, since the same decomposition in equation \ref{eq:weightdecomposition} holds in the quantum case due to Lemma \ref{le:flat}. By Section \ref{se:additional}, we know that the projection
\[ \dot{U}_q(\mathfrak{gl}(m))/I_{(K,0,\ldots,0)}\cong \bigoplus_{l\in L((K,0,\ldots,0))}\operatorname{End}_{\C(q)}(V_m(l)) \to \bigoplus_{\mu\in H,|\mu|=K} \operatorname{End}_{\C(q)}(V_m(\mu)) \]
descends to a surjection
\[ \frac{\dot{U}_q(\mathfrak{gl}(m))}{I_{(K,0,\ldots,0)}+\Xi} \to \bigoplus_{\mu\in H,|\mu|=K} \operatorname{End}_{\C(q)}(V_m(\mu)). \]

We wish to show this is injective. For this it suffices to show that for each $\nu\leq (K,0,\ldots,0)$ with $\nu\not\in H$, the projection to $\operatorname{End}_{\C(q)}(V_m(\nu))$ belongs to $\Xi$. In fact, since the matrix algebra $\operatorname{End}_{\C(q)}(V_m(\nu))$ is simple, it is enough to show there is an element $r\in \Xi$ such that $r$ does not vanish on $V_m(\nu)$.

Let $\nu$ be a dominant weight with $\nu_2>1$, and suppose $\nu_1=k$, $\nu_2=l$. Then the element
\[ \begin{tikzpicture}[baseline=-0.65ex,scale=0.5]
\draw (-0.5,-1) -- (-0.5,1);
\draw (0.5,-1) -- (0.5,1);
\draw (1.5,0) node {$\cdots$};
\coordinate[label=left:$k$] (k) at (-0.5,0.5);
\coordinate[label=right:$l$] (l) at (0.5,0.5);
\end{tikzpicture} \]
is the identity on $V_m(\nu)$, and so the element of $\Xi$ shown in equation \ref{eq:nonmoy2} with $m-2$ uprights added to the right-hand side acts as $[k+1][k][l][l-1]$ on $V_m(\nu)$ as both other terms factor through weights higher than $\nu$, and so act as $0$ on $V_m(\nu)$. Hence this exhibits an element of $\Xi$ acting non-trivially on $V_m(\nu)$, and so the claim is proved.
\end{proof}

\begin{cor}\label{co:threecategories}
There is a full functor \[\bigoplus_m \operatorname{Lad}_m \to \operatorname{Rep}\]
and a fully faithful functor
\[\bigoplus_m \operatorname{Lad}^\Xi_m \to \operatorname{Rep}. \]
\end{cor}
\begin{proof}
This follows directly from Corollary \ref{co:full}, Corollary \ref{co:lad} and Theorem \ref{th:faithful}.
\end{proof}

It is also interesting to note the following:

\begin{cor}\label{co:equivalent}
The non-MOY relation in Section \ref{se:additional} is equivalent to the relation in equation \ref{eq:nonmoy2}.
\end{cor}
\begin{proof}
We have seen that the non-MOY relation implies the derived relation. To see the converse, we note that the proof of Theorem \ref{th:faithful} only required the derived relation. It thus follows that the ladder relations with the derived relation are also equivalent to relations in $\operatorname{Rep}$, and since the non-MOY relation is a relation in $\operatorname{Rep}$, it must be a consequence of the other relations.
\end{proof}

Since ladders are trivalent graphs we can interpret them as MOY diagrams. One then notices that the relations on ladder diagrams agree with the relations on MOY diagrams in Section \ref{se:moymoves} where both are defined. We note in particular that the resolutions of braid diagrams can always be written as ladders, for which we now have several nice interpretations.

\subsection{Braiding}\label{se:braiding}
There is a braiding on $\operatorname{Rep}$ defined by the $R$-matrix of $U_q(\mathfrak{gl}(1|1))$. Since we now have full functors $\operatorname{Rep}$ from the equivalent categories $\dot{U}^\infty_q(\mathfrak{gl}_\bullet)=\bigoplus_m(\dot{U}^\infty_q(\mathfrak{gl}(m)))$ and $\operatorname{Lad}=\bigoplus_m \operatorname{Lad}_m$, it is interesting to ask if the braiding occurs in these other categories.

Recall that a braided monoidal category is a monoidal category with a natural isomorphism from the bifunctor $-\otimes -$ to $-\otimes^{op}-$ (ie. a natural system of isomorphisms $\beta_{U,V}:U\otimes V \to V\otimes U$), satisfying the two equations
\[ \beta_{U\otimes V,W}=(\beta_{U,W}\otimes \id_V)\circ (\id_U\otimes \beta_{V,W}) \]
\[ \beta_{U,V\otimes W}= (\id_V\otimes \beta_{U,W})\circ(\beta_{U,V}\otimes \id_W) \]
for any $U,V,W$, called the hexagon equations.

Following \cite{Cautis2012} and \cite{Lusztig1993}, we can define the morphism
\[ 1_{s_i(k)}T_i1_k=(-1)^{k_ik_{i+1}} \sum_{\mathclap{\substack{r,s\geq 0 \\ r-s=k_i-k_{i+1}}}} (-q)^{k_{i}-s} E_i^{(s)}F_i^{(r)} 1_k \in 1_{s_i(k)}\dot{U}^\infty_q(\mathfrak{gl}(m))1_k \]
where $s_i(k)=(k_1,\ldots,k_{i+1},k_i,\ldots,k_m)$ with the $i$th and $(i+1)$th coordinates swapped. Note that this sum is actually finite, since $F_i^{(r)}1_{k}=0$ for large $r$.

By \cite[Section~39]{Lusztig1993}, these elements satisfy
\[ T_iT_{i+1}T_i1_k=T_{i+1}T_iT_{i+1}1_k \]
\[ T_i T_j1_k = T_j T_i 1_k \quad \text{if } |i-j|>2 \]
which are the braid relations. Also, each $T_i$ is invertible.

If we define the structure of a monoidal category on $\dot{U}^\infty_q(\mathfrak{gl}_\bullet)$ by setting
\[ 1_{(k_1,\ldots,k_s)}\otimes 1_{(l_1,\ldots,l_t)} := 1_{(k_1,\ldots,k_s,l_1,\ldots,l_t)} \]
then we have the following
\begin{thm}\label{th:braiding}
The $T_i$ define the structure of a braided monoidal category on $\dot{U}^\infty_q(\mathfrak{gl}_\bullet)$.
\end{thm}
\begin{proof}
The $T_i$ define a map $1_{k}\to 1_{s_i(k)}$, which transposes two entries in $k$. For any permutation $w\in S_m$, we can define an element $T_w$ taking $1_k \to 1_{w(k)}$ by composing the tranpositions $T_i$ minimising word length. So if we define the braiding between $1_k$ and $1_l$ to be
\[ T_w 1_{(k_1,\ldots,k_s,l_1,\ldots,l_t)} \]
where $w\in S_{s+t}$ is the permutation that swaps the first $s$ and last $t$ entries, sending $(k_1,\ldots,k_s,l_1,\ldots,l_t)$ to $(l_1,\ldots,l_t,k_1,\ldots,k_s)$, then a straightforward generalisation of
\cite[Lemma~6.1.3]{Cautis2012} and \cite[Theorem~6.1.4]{Cautis2012} gives the result.
\end{proof}

\begin{thm}
The functor $\phi:\dot{U}^\infty_q(\mathfrak{gl}_\bullet) \to \operatorname{Rep}$ preserves the braiding.
\end{thm}
\begin{proof}
This is essentially the argument from \cite{Cautis2012}. A straightforward verification from Section \ref{se:representationtheory} shows that $\phi(T1_{(1,1)})=R_{\C_q^{1|1},\C_q^{1|1}}$, the $R$-matrix $\C_q^{1|1}\otimes \C_q^{1|1}\to \C_q^{1|1}\otimes \C^{1|1}_q$. It follows from the hexagon relations that $\phi(T_w1_{(1^k,1^l)})=R_{(\C_q^{1|1})^{\otimes k},(\C_q^{1|1})^{\otimes l}}$ where $1^k=(1,\ldots,1)$ has $k$ terms. To determine the braiding on $\bigwedge^k_q (\C_q^{1|1})$ and $\bigwedge^l_q(\C_q^{1|1})$, we use the injective map
\[ \psi_k:=\phi(F_1F_2F_1F_3F_2F_1\cdots F_{k-1}F_{k-2}\cdots F_1 1_{(k,0^{k-1})}):\bigwedge^k_q (\C_q^{1|1}) \to (\C_q^{1|1})^{\otimes k}.\]
By naturality of the braidings, we have
\[ R_{((\C_q^{1|1})^{\otimes k},(\C_q^{1|1})^{\otimes l})}\circ (\psi_k\otimes \psi_l)= (\psi_l\otimes \psi_k)\circ R_{(\bigwedge_q^k \C_q^{1|1},\bigwedge_q^l \C_q^{1|1})} \]
and
\[ \phi(T_w1_{(1^k,1^l)})\circ (\psi_k\otimes \psi_l) = (\psi_l\otimes \psi_k)\circ \phi(T_w 1_{((k,0^{k-1}),(l,0^{l-1}))}). \]
Since $\phi(T_w1_{(1^k,1^l)})=R_{(\C_q^{1|1})^{\otimes k},(\C_q^{1|1})^{\otimes l}}$, the left-hand terms of these two equations are equal, and therefore so are the right-hand terms. Since $\psi_l\otimes \psi_k$ is injective, \[\phi(T_w 1_{((k,0^{k-1}),(l,0^{l-1}))}) =   R_{(\bigwedge_q^k \C_q^{1|1},\bigwedge_q^l \C_q^{1|1})} \]
and by the hexagon relations it follows that the braidings are equal on all of $\operatorname{Rep}$.
\end{proof}

Of course, it is possible to define the same braiding on $\dot{U}^\Xi_q(\mathfrak{gl}_\bullet)$ in the obvious way, and then the functor $\dot{U}^\Xi_q(\mathfrak{gl}_\bullet) \to \operatorname{Rep}$ also preserves the braiding by definition. This shows that the equivalence of $\operatorname{Rep}$ with our ladder category $\bigoplus_m\operatorname{Lad}^\Xi_m$ allows us to express the $R$-matrix in a simple way using sums of a ladder diagrams as defined by the $T_i$. In this way, we recover the expression for the braiding on strands coloured with the standard representation of $\dot U_q(\mathfrak{gl}(1|1)$ as given in figure \ref{fi:knot}, and the braiding for higher colours is similarly given by
\[\begin{tikzpicture}[baseline=-0.65ex]
\coordinate [label={below:$k_1$}] (f) at (-1,-1);
\coordinate [label={below:$k_2$}] (s) at (1,-1);
\draw[->] (-1,-1) -- (1,1);
\draw (1,-1) -- (0.25, -0.25);
\draw[<-] (-1,1) -- (-0.25,0.25);
\end{tikzpicture} = (-1)^{k_1k_2}\sum_{\mathclap{\substack{r,s\geq 0 \\ r-s=k_1-k_2}}} (-q)^{k_2-s}\left(\begin{tikzpicture}[baseline=-0.65ex]
\draw (-0.5,-1) -- (-0.5,1);
\draw (0.5,-1) -- (0.5,1);
\draw (-0.5,-0.6) -- (0.5,-0.2);
\draw (0.5,0.2) -- (-0.5,0.6);
\draw (-0.5,-1.25) node {$k_1$};
\draw (0.5,-1.25) node {$k_2$};
\draw (-0.5,1.25) node {$k_2$};
\draw (0.5,1.25) node {$k_1$};
\draw (0,-0.6) node {$r$};
\draw (0,0.6) node {$s$};
\draw (-1.1,0) node {$k_1-r$};
\draw (1.1,0) node {$k_2+r$};
\end{tikzpicture}\right) \]
just by using the expression for $T_11_{(k_1,k_2)}$ in ladder form.

\printbibliography
\end{document}